\documentclass[11pt]{article}
\usepackage[scale=0.75]{geometry}
\usepackage[T1]{fontenc}
\usepackage[utf8]{inputenc}
\usepackage{amsmath,amssymb,amsthm,mathtools}
\usepackage[authoryear,round]{natbib}

\usepackage[colorlinks,citecolor=blue,urlcolor=blue, linkcolor = blue]{hyperref}
\usepackage{cleveref}
\usepackage{graphicx}
\usepackage{xcolor}
\usepackage{comment}
\usepackage{enumerate}

\numberwithin{equation}{section}
\swapnumbers

\newcommand{\E}{\mathbb{E}}
\newcommand{\Prob}{\mathbb{P}}
\newcommand{\R}{\mathbb{R}}

\DeclareMathOperator*{\argmax}{argmax}

\theoremstyle{plain}

\newtheorem{theorem}{Theorem}[section]
\newtheorem{lemma}[theorem]{Lemma}
\newtheorem{proposition}[theorem]{Proposition}

\theoremstyle{definition}

\newtheorem{assumption}[theorem]{Assumption}

\theoremstyle{remark}
\newtheorem{remark}[theorem]{Remark}

\title{Parameter estimation in a fully coupled partially observed Ornstein--Uhlenbeck process}

\author{
Sascha Gaudlitz\footnote{Humboldt--Universität zu Berlin \\ Email: sascha.gaudlitz<at>hu-berlin.de}
\quad \& \quad
Hasan Mert Gökalp\footnote{Technische Universität Berlin \\ Email: goekalp<at>math.tu-berlin.de}
}
\date{}
\begin{document}

\maketitle

\begin{abstract}
We study a two-dimensional Ornstein--Uhlenbeck system where only the first coordinate is observed, whereas the second coordinate remains hidden. Our goal is the estimation of the coupling parameter in the drift of the observed coordinate.
The core novelty lies in accounting for the influence of the observed component on the unobserved one, making the system fully coupled.
Using linear filtering, we derive the likelihood under partial observation and establish local asymptotic normality of the statistical model. Within the Ibragimov–Hasminskii framework (1981), we prove consistency, asymptotic normality, convergence of moments and asymptotic efficiency of the MLE under stability and identifiability assumptions as the time horizon tends to infinity. 
\end{abstract}



\section{Introduction}\label{sec:intro}

We consider the fully coupled linear system
\begin{subequations}\label{eq:XYsystem}
  \begin{align}
    dX_t & = \bigl(a X_t + \vartheta Y_t\bigr)dt + dW_t^{X}, \label{eq:Xeq} \\
    dY_t & = \bigl(b X_t + c Y_t\bigr)dt + dW_t^{Y}, \label{eq:Yeq}
  \end{align}
\end{subequations}
for $t\in[0,T]$ with $X_0=Y_0=0$, where $W_{t}^{X}$ and $W_{t}^{Y}$ are independent Brownian motions on a filtered probability space $(\Omega, \mathcal{F}, \allowbreak (\mathcal{F}_t)_{t\ge0}, \Prob)$. The real scalars $a$, $b$ and $c$ are known, whereas $\vartheta$ shall be estimated in the large time asymptotic $T\to\infty$.
We assume that only the process $X^T= (X_t, 0\le t \le T)$ is observed, while $Y^T = (Y_t, 0\le t \le T)$ is hidden.
This partial observation scheme naturally relates our setting to the frameworks of hidden Markov processes \citep*{cappeInferenceHiddenMarkov2005} and Kalman-Bucy filtering \citep{Bain2009}, where $X^T$ is commonly viewed as a noisy observation of the underlying (hidden) dynamic $Y^T$. In these settings, the process $X^T$ typically has no effect on $Y^T$, i.e. $b=0$ in \eqref{eq:XYsystem}.
In contrast, our setting models the underlying dynamic by the fully coupled two-dimensional system $(X^T,Y^T)$ itself, while only measurements from one component $X^T$ are available to the statistician.
Partially observable fully coupled systems arise naturally in many applications, including the FitzHugh-Nagumo equations in computational neuroscience \citep{fitzhughImpulsesPhysiologicalStates1961} and the dormancy phenomenon in population biology \citep{lennonPrinciplesSeedBanks2021,blathNewCoalescentSeedbank2016a}.
While these systems often involve significant non-linearities, the idealised model \eqref{eq:XYsystem} isolates the fundamental challenge of inference in a partially observed fully coupled system.

Statistical inference in the model \eqref{eq:XYsystem} is well-studied when there is no feedback from $X^T$ to $Y^T$, such that the process $Y^T$ evolves autonomously from $X^T$ and the filter equations can be analysed decoupled from $X^T$. Results in this setting can be found in \citep{Kutoyants1984,KALLIANPUR1991284,kutoyants2004statistical,Chigansky_2008,kutoyantsParameterEstimationHidden2019} for $T\to\infty$, and in \citep{KUTOYANTS2019248,kutoyantsParameterEstimationHidden2021} for vanishing observational noise. \citet{Brouste_2010} study the setting of fractional noise. Online and recursive estimation is considered by \citet{suraceOnlineMaximumLikelihoodEstimation2019,moura1986,elliottExactFiniteDimensionalFilters1997,demboParameterEstimationPartially1986}. Discrete systems have been considered for example by \citet{papavasiliouParameterEstimationAsymptotic2006,mongilloOnlineLearningHidden2008,kurisakiParameterEstimationErgodic2023}.

Instead, we are interested in estimating $\vartheta$ when both components interact $(b\neq0)$. 
In this case, the filter equations are coupled to $X^T$ and the statistical analysis requires a careful study of the joint process composed of $X^T$ and the filters.
We utilise the ergodicity of the joint process to obtain the convergence of time averages in probability, the convergence of the covariance matrices, and exponential decay of autocovariances in Proposition \ref{prop:aou-propositions}.
With this concentration result at hand, we can prove the local asymptotic normality (LAN) of the statistical experiment in Theorem \ref{thm:lan_property}, establish that the scaled likelihood ratio process is Lipschitz continuous with respect to the Hellinger distance, and that the Hellinger affinity \citep[Chapter 4]{LeCam1986} between the measures corresponding to the true and shifted parameters decays exponentially fast.
Subsequently, we follow the general framework of \citet{IbragimovHasminskii1981} to obtain consistency, asymptotic normality, efficiency and convergence of moments for the MLE for $\vartheta$ in Theorem \ref{thm:main-mle}. Furthermore, we quantify the information loss compared to observing the full system $(X^T,Y^T)$ and find that this vanishes as the influence of $X^T$ on $Y^T$ grows, i.e. $|b|\to\infty$.
Finally, we discuss the implications of our analysis for parameter estimation in partially observed linear stochastic partial differential equations (SPDEs).

The paper is structured as follows. In Section~\ref{sec:setup} we present the general setting, filter equations, establish the ergodic properties of the joint process with the filters and derive the Fisher information. In Section~\ref{sec:main-results} we establish the LAN property and prove consistency, asymptotic normality, efficiency and convergence of moments for the MLE within the framework of \citet{IbragimovHasminskii1981}. In Section \ref{sec:Perspectives} we discuss extensions of the model to linear SPDEs which can be seen as infinite-dimensional Ornstein--Uhlenbeck processes. Technical proofs are given in Appendix~\ref{sec:tech-proofs}.

\section{Preliminaries}\label{sec:setup}
In this section, we introduce the likelihood, the filter equations and the Fisher information. We also establish the long time behaviour of the joint process composed of $X^T$ and the filters.

\subsection{Setting}

We denote by $(\Prob_{\vartheta}^T:\vartheta\in\mathbb{R})$ the family of laws induced by the observations \eqref{eq:Xeq} on the space of continuous functions $\bigl(C([0,T];\mathbb R),\mathcal B(C([0,T];\mathbb R))\bigr)$. We denote by $\xrightarrow{d}$ convergence in distribution and by $\xrightarrow{\Prob_{\vartheta}^T}$ convergence in probability. Consider an open, non-empty and bounded parameter space $\Theta\subset \mathbb{R}$ with closure $\bar{\Theta}$. Throughout, the data is generated under a true value $\vartheta_0\in\Theta$, such that $X^T\sim \Prob_{\vartheta_0}^T$. If $T$ is clear from the context, we write $\mathbb{P}_\vartheta$ instead of $\mathbb{P}_\vartheta^T$. Throughout the paper, $C$ denotes a positive finite constant, independent of $T$ and of the parameters, whose value may change from line to line.

We first employ Theorem 7.17 of \citet{LiptserShiryaev2001} to obtain an adapted representation of the stochastic differential equation (SDE) for $X^T$ in \eqref{eq:Xeq} as
\begin{equation}\label{eq:Xadapted}
  dX_t = \bigl(aX_t + \vartheta \mathbb{E}_{\vartheta}\bigl[Y_t \mid \mathcal{F}_t^X\bigr]\bigr) dt + d\bar{W}_t^{\vartheta},\quad t\in[0,T],
\end{equation} where $(\mathcal{F}_t^X)_{t\in[0,T]}$ is the filtration generated by $(X_s)_{0 \le s \le t}$ for $t\in[0,T]$, and $(\bar{W}_t^{\vartheta})_{t\in[0,T]}$ is an $(\mathcal{F}_t^X)_{t\in[0,T]}$-Brownian motion under $\Prob_\vartheta$, called the innovation process. When $\vartheta$ is clear from the context, we write $\bar W$ instead of $\bar W^\vartheta$. Since the drift in \eqref{eq:Xadapted} is $(\mathcal{F}_t^X)_{t\in[0,T]}$-adapted, we can apply Theorem 7.19 of \citet{LiptserShiryaev2001} to obtain the equivalence of the measures
$(\Prob_\vartheta^T:\vartheta\in\Theta)$. The likelihood of $\vartheta$ based on the observation $X^T$, with respect to the dominating measure $\Prob_0^T$, is given by
\begin{equation}\label{eq:likelihood_X}
  \begin{aligned}
     L_T(\vartheta)
    \coloneq \exp\Biggl(
      \vartheta\int_0^T \mathbb{E}_\vartheta\bigl[Y_t \mid \mathcal{F}_t^X\bigr]\left( dX_t - aX_t dt\right)
    - \frac{1}{2} \vartheta^2 \int_0^T
    \mathbb{E}_\vartheta\bigl[Y_t \mid \mathcal{F}_t^X\bigr]^2  dt
    \Biggr).
  \end{aligned}
\end{equation}
We define the MLE as
\[\hat{\vartheta}_T \in \argmax_{\vartheta \in \bar \Theta} L_T(\vartheta).\]
The subsequent analysis also implies that $\vartheta\mapsto\mathbb{E}_{\vartheta}\bigl[Y_t \mid \mathcal{F}_t^X\bigr]$ is almost surely continuous (see Lemma \ref{lem:m_exp_bounds}), such that a maximiser exists. Note that it might not be unique. We also introduce the likelihood ratio
$L_T(\vartheta_1;\vartheta_2)\coloneqq L_T(\vartheta_1)/L_T(\vartheta_2)$ for $\vartheta_1,\vartheta_2\in\Theta$.

\begin{assumption}\label{ass:param_space}
  The model coefficients satisfy $a+c<0$ and $b\neq 0$. Moreover, the condition
  \[
    ac - b\vartheta > 0
  \] holds for every $\vartheta\in\bar{\Theta}$.
\end{assumption}

  Assumption \ref{ass:param_space} is imposed throughout the paper and ensures that the drift matrix
  of the two-dimensional model in
  \eqref{eq:XYsystem} is Hurwitz in the sense that all of its eigenvalues have negative real parts. This assumption forces $a$ and $c$ to be
  negative if $0\in\bar\Theta$.

\subsection{The filter equations}

For the linear system  \eqref{eq:XYsystem}, the conditional distribution of $Y_t$ given $\mathcal{F}_t^X$ is normal and uniquely determined by its mean and covariance  \citep[Lemma 6.12]{Bain2009}. For $t\in[0,T]$ and $\vartheta\in\Theta$ we define the conditional mean $ m_t(\vartheta) \coloneq \mathbb{E}_\vartheta\bigl[Y_t \mid \mathcal{F}_t^X\bigr] $ and conditional variance $ \gamma_t(\vartheta) \coloneq  \mathbb{E}_\vartheta\bigl[(Y_t-m_t(\vartheta))^2\mid  \mathcal{F}_t^X\bigr]$. By Theorem 10.3 of \citet{LiptserShiryaev2001}, they satisfy
\begin{align}
  dm_t(\vartheta)
   & = \left(cm_t(\vartheta) + bX_t\right)dt
  + \gamma_t(\vartheta)\vartheta
  \left(dX_t - \bigl(\vartheta m_t(\vartheta) + aX_t\bigr)dt\right),
  \label{eq:filter-eq}                                                   \\
  \frac{d\gamma_t(\vartheta)}{dt}
   & = 2c\gamma_t(\vartheta) + 1 - \vartheta^2\gamma_t^2(\vartheta),
  \label{eq:cond_var}
\end{align}
with the initial conditions $ m_0(\vartheta) = \mathbb{E}_\vartheta[Y_0 | X_0] = 0$ and $\gamma_0(\vartheta) = \mathbb{E}_\vartheta[(Y_0-m_0)^2|X_0] = 0$.

Note that under $\Prob_\vartheta$, $dX_t-(\vartheta m_t(\vartheta)+aX_t)dt=d\bar W_t^\vartheta$ as in \eqref{eq:Xadapted} and that $(m_t(\vartheta),\gamma_t(\vartheta))$ is the true filter only under $\Prob_\vartheta$. Since the system is Gaussian, the conditional variance $\gamma_t(\vartheta)$ is deterministic. Equation \eqref{eq:cond_var} is a Riccati equation with constant coefficients and admits an explicit solution for $\vartheta \neq0$ given by
\begin{equation} \label{eq:gamma_exp}
  \gamma_t(\vartheta)
  = \frac{c+\sqrt{c^2+\vartheta^2}}{\vartheta^2}
  - \frac{2\sqrt{c^2+\vartheta^2}}
  {(c-\sqrt{c^2+\vartheta^2})^2\exp\big(2t\sqrt{c^2+\vartheta^2}\big)+\vartheta^2}.
\end{equation}
If $0\in\bar\Theta$, the corresponding solution at $\vartheta=0$ is given by $\gamma_t(0)=(\exp(2ct)-1)/(2c)$, which coincides with the continuous extension of the explicit formula in \eqref{eq:gamma_exp}.


\begin{remark}\label{rem:on_filter}~
\begin{enumerate}[(i)]
  \item In the non-interacting case $b=0$, we observe $\mathbb{E}_{\vartheta}\bigl[Y_t \mid \mathcal{F}_t^X\bigr] = - \mathbb{E}_{-\vartheta}\bigl[Y_t \mid \mathcal{F}_t^X\bigr]$ (see \eqref{eq:filter-eq}) and deduce $L_T(\vartheta) = L_T(-\vartheta)$. Therefore, $\vartheta$ is not identifiable from the observation of $X^T$. This identifiability issue disappears if we restrict the parameter space $\Theta$ to negative or positive numbers.
\item \label{num:gamma_conv}
  Note that $\gamma_t(\vartheta)$ converges uniformly and exponentially fast to $\gamma_{\infty}(\vartheta)\coloneq\frac{c+\sqrt{c^2+\vartheta^2}}{\vartheta^2}$ on the parameter space $\Theta$.
\item Our results rely on the linear structure of the model \eqref{eq:XYsystem}, which allows us to derive an explicit SDE for the conditional mean in \eqref{eq:filter-eq}. This is not possible for general nonlinear systems, where the conditional distribution of the hidden process given the observed process typically does not stay within a finite-dimensional family like the Gaussian family. Therefore, one generally cannot obtain a closed-form SDE for the conditional mean.
\end{enumerate}
\end{remark}

To analyse the asymptotic properties of the MLE, we need to compute the derivatives of $m_t$ and $\gamma_t $ with respect to $\vartheta$. The derivative $\partial_\vartheta \gamma_t(\vartheta)$ is formally obtained by differentiating the Riccati equation in \eqref{eq:cond_var} \citep[Theorem 2.11]{Teschl2012-ODE}. The mean square derivative $\partial_\vartheta m_t(\vartheta)$ of $m_t$ with respect to $\vartheta$ satisfies the linear stochastic differential equation obtained via a formal differentiation of \eqref{eq:filter-eq}. Similarly to \citet{KUTOYANTS2019248} we obtain the following result.

\begin{lemma}
  \label{lem:mean-square-diff}
  The conditional mean $\Theta\ni\vartheta\mapsto m_t(\vartheta)$ is differentiable in mean square with a quadratic remainder estimate under each $\Prob_{\vartheta_0}$, $\vartheta_0\in\Theta$, for any $t\in[0,T]$ with derivative $(\partial_\vartheta m_t(\vartheta))_{t\in[0,T]}$ satisfying the SDE
\begin{equation}\label{eq:filter_Derivative}
\begin{aligned}
d(\partial_\vartheta m_t(\vartheta))
&=
\Bigl((c-\vartheta^2\gamma_t(\vartheta))\partial_\vartheta m_t(\vartheta)
-\vartheta\gamma_t(\vartheta)m_t(\vartheta)\Bigr)dt \\
&\quad+
\Bigl(\gamma_t(\vartheta)+\vartheta\partial_\vartheta\gamma_t(\vartheta)\Bigr)
\Bigl(dX_t-(\vartheta m_t(\vartheta)+aX_t)dt\Bigr).
\end{aligned}
\end{equation}
with initial condition $\partial_\vartheta m_0(\vartheta)=0$,
in the sense that there exists a constant $C>0$ such that
\begin{equation}\label{eq:bound_rest}
    \mathbb{E}_{\vartheta_0}\left[\left(\frac{1}{h} (m_t(\vartheta+h) - m_t(\vartheta)) - \partial_\vartheta m_t(\vartheta)\right)^2\right] \le Ch^2.
  \end{equation}
  for any $\vartheta,\vartheta_0\in\Theta$, $t\in[0,T]$ and $h$ such that $\vartheta+h\in\Theta$.
\end{lemma}
\noindent The \hyperref[prf:m_twice_diff]{Proof of Lemma~\ref{lem:mean-square-diff}} is given in Section \ref{prf:m_twice_diff}.

\subsection{Ergodic properties}
To study the long-time behaviour of the filter process \eqref{eq:filter-eq} and its derivative \eqref{eq:filter_Derivative}, we have to account for the feedback term $bX_t$, which couples the filter dynamics to the observed component. We therefore consider the augmented process $(X_t,Y_t,m_t(\vartheta),\partial_\vartheta m_t(\vartheta))_{t\ge 0}$. Including $Y_t$ will also allow us to use the same ergodic framework when comparing partial and full information in Section~\ref{sec:info-loss}. Under $\Prob_{\vartheta_0}$, the equations for $X_t$, $Y_t$, $m_t(\vartheta)$ and $\partial_\vartheta m_t(\vartheta)$ form a linear stochastic differential system with time-dependent coefficients. In particular, the system $(X_t,Y_t,m_t(\vartheta),\partial_\vartheta m_t(\vartheta))$ is well-defined for any $t\ge 0$. The exponentially fast convergence of $\gamma_t(\vartheta)$ and $\partial_\vartheta\gamma_t(\vartheta)$ to constants directly transfers to the time-dependent coefficients of $(X_t,Y_t,m_t(\vartheta),\partial_\vartheta m_t(\vartheta))_{t\in[0,T]}$. In combination with the Hurwitz condition provided by Assumption \ref{ass:param_space}, we next establish the asymptotic concentration for a general class of (linear) SDEs.

\begin{proposition} \label{prop:aou-propositions}
For $d,p\in\mathbb{N}$ let $\bigl(Z_t(\vartheta_0,\vartheta)\bigr)_{\vartheta_0,\vartheta\in\Theta}$ be a family of $d$-dimensional stochastic processes such that, for each $(\vartheta_0,\vartheta)\in\Theta^2\coloneq \Theta\times \Theta$, the process $Z_t(\vartheta_0,\vartheta)$ satisfies the SDE
\[
    dZ_t(\vartheta_0,\vartheta)
    = A(t,\vartheta_0,\vartheta)Z_t(\vartheta_0,\vartheta)dt
    + B(t,\vartheta_0,\vartheta)d\mathcal{W}_t,\quad t\ge 0,
\] 
$Z_0=0\in\mathbb{R}^d$, and a $p$-dimensional Wiener process $(\mathcal{W}_t)_{t\ge 0}$. The matrices $A(t,\vartheta_0,\vartheta)\in\mathbb{R}^{d\times d}$ and $B(t,\vartheta_0,\vartheta)\in\mathbb{R}^{d\times p}$ are assumed to be continuous in $(\vartheta_0,\vartheta)$ and admit continuous extensions to $\bar\Theta^{2}$. Assume that, for each $(\vartheta_0,\vartheta)\in\bar\Theta^{2}\coloneq \bar{\Theta}\times \bar{\Theta}$, the limits
  \[
    A(\vartheta_0,\vartheta)\coloneq\lim_{t\to\infty}A(t,\vartheta_0,\vartheta)\in\mathbb{R}^{d\times d},
    \qquad
    B(\vartheta_0,\vartheta)\coloneq\lim_{t\to\infty}B(t,\vartheta_0,\vartheta)\in\mathbb{R}^{d\times p}
  \]
exist and that the following conditions hold for all $\vartheta_0,\vartheta\in\bar\Theta$ (for some matrix norm $\|\cdot\|$):

\begin{enumerate}
    \item Uniform Hurwitz condition / exponential stability: There exist constants $0<M,\beta<\infty$ independent of $\vartheta_0 \text{ and }\vartheta$ such that
          \[
            \bigl\|\exp\bigl(tA(\vartheta_0,\vartheta)\bigr)\bigr\| \le Me^{-\beta t}\quad\text{for all }t\ge 0.
          \]
    \item Exponential convergence of coefficients: There exist constants $0<C,\alpha<\infty$ independent of $\vartheta_0 \text{ and }\vartheta$ such that
          \[
            \bigl\|A(t,\vartheta_0,\vartheta)-A(\vartheta_0,\vartheta)\bigr\|
            +\bigl\|B(t,\vartheta_0,\vartheta)-B(\vartheta_0,\vartheta)\bigr\|
            \le C e^{-\alpha t}\quad\text{for all }t\ge 0.
          \]
    \item Uniform boundedness of the diffusion: There exists a constant $0<\bar B<\infty$ independent of $\vartheta_0 \text{ and }\vartheta$ such that
          \[\sup_{t\ge0}\|B(t,\vartheta_0,\vartheta)\|\le \bar B.\]
\end{enumerate}
For each $(\vartheta_0,\vartheta)\in\bar\Theta^{2}$, let $\bar Z_t(\vartheta_0,\vartheta)$ denote the Ornstein--Uhlenbeck process started in zero and driven by $(\mathcal{W}_t)_{t\ge 0}$ with constant coefficient matrices $A(\vartheta_0,\vartheta)$ and $B(\vartheta_0,\vartheta)$. Denoting its invariant Gaussian law by $\mu_{\vartheta_0,\vartheta}$ we define for 
$\bar Z(\vartheta_0,\vartheta)\sim\mu_{\vartheta_0,\vartheta}$ 
\[
  \Sigma(\vartheta_0,\vartheta)
  \coloneq
  \E_{\mu_{\vartheta_0,\vartheta}}\left[
    \bar Z(\vartheta_0,\vartheta)
    \bar Z(\vartheta_0,\vartheta)^\top
  \right].
\]
Then the following statements hold. 
\begin{enumerate}
\renewcommand{\labelenumi}{\alph{enumi})}
\renewcommand{\theenumi}{\alph{enumi}}
        \item \label{item:aou-b} Define
          \[
            \Sigma(t,\vartheta_0,\vartheta)\coloneqq
            \mathbb{E}_{\vartheta_0}\big[
              Z_t(\vartheta_0,\vartheta)Z_t(\vartheta_0,\vartheta)^\top
            \big],\quad t\ge 0.
          \]
          Then there exist constants $0<C,\gamma<\infty$ such that
          \[
            \sup_{\vartheta_0,\vartheta\in \bar \Theta}
            \bigl\|\Sigma(t,\vartheta_0,\vartheta)-\Sigma(\vartheta_0,\vartheta)\bigr\|
            \le Ce^{-\gamma t}
            \qquad t\ge0.
          \]
          Moreover, the map $(\vartheta_0,\vartheta)\mapsto \Sigma(\vartheta_0,\vartheta)$ is continuous on $\bar\Theta^{2}$.

    \item \label{item:aou-c} There exists a constant $0<C<\infty$ such that
          \[
            \sup_{\vartheta_0,\vartheta\in \bar \Theta}
            \bigl\|\operatorname{Cov}_{\vartheta_0}\bigl(
              Z_t(\vartheta_0,\vartheta),Z_s(\vartheta_0,\vartheta)
            \bigr)\bigr\|
            \le
            Ce^{-\beta|t-s|}
            \qquad t,s\ge0.
          \]
    \item \label{item:aou-a} Let $h$ be a quadratic function of the form
          \[
            h(z)= z^\top \mathcal{M} z + \mathcal{K}^\top z + \mathcal{C},\quad z\in\mathbb{R}^d,
          \]
          with $\mathcal{M}\in\mathbb{R}^{d\times d}$ a symmetric matrix, $\mathcal{K}\in\mathbb{R}^d$ and $\mathcal{C}\in\mathbb{R}$.
          Then the time averages of $h\bigl(Z_t(\vartheta_0,\vartheta)\bigr)$ converge to the corresponding stationary expectation in $L^1(\Prob_{\vartheta_0})$, uniformly over $(\vartheta_0,\vartheta)\in\bar\Theta^2$, that is
          \begin{equation*}
            \lim_{T\to\infty} \sup_{\vartheta_0,\vartheta\in\bar{\Theta}} \E_{\vartheta_0}\left[\left|\frac{1}{T}\int_0^T h(Z_s(\vartheta_0,\vartheta)) ds - \E_{\mu_{\vartheta_0,\vartheta}}\bigl[h(\bar Z(\vartheta_0,\vartheta))\bigr]
            \right|\right]=0.
        \end{equation*}
  \end{enumerate}
\end{proposition}

\noindent The proof of Proposition \hyperref[prf:aou-prop]{\ref{prop:aou-propositions}} can be found in Section \ref{prf:aou-prop}.

\begin{remark}\label{rem:stable-process}
  For each $\vartheta_0\in\Theta$, define
$Z_t(\vartheta_0)\coloneq\bigl(X_t,Y_t,m_t(\vartheta_0),\partial_\vartheta m_t(\vartheta_0)\bigr)^\top$, which satisfies the linear SDE
  \begin{equation}\label{eq:full_proc}
    dZ_t(\vartheta_0)=A(t,\vartheta_0)Z_t(\vartheta_0)dt+B(t,\vartheta_0)d\mathcal{W}_t,
    \qquad
    d\mathcal{W}_t=\begin{pmatrix}dW_t^X\\ dW_t^Y\end{pmatrix},
  \end{equation}
where
\begin{align*}
\setlength{\arraycolsep}{2.5pt}
A(t,\vartheta_0)&=
\begin{pmatrix}
  a & \vartheta_0 & 0 & 0 \\
  b & c & 0 & 0 \\
  b & \vartheta_0^2\gamma_t(\vartheta_0) & c-\vartheta_0^2\gamma_t(\vartheta_0) & 0 \\
  0 & \vartheta_0\gamma_t(\vartheta_0)+\vartheta_0^2\partial_\vartheta\gamma_t(\vartheta_0)
    & -2\vartheta_0\gamma_t(\vartheta_0)-\vartheta_0^2\partial_\vartheta\gamma_t(\vartheta_0)
    & c-\vartheta_0^2\gamma_t(\vartheta_0)
\end{pmatrix},\\
B(t,\vartheta_0)&=
\begin{pmatrix}
  1 & 0 \\
  0 & 1 \\
  \vartheta_0\gamma_t(\vartheta_0) & 0 \\
  \gamma_t(\vartheta_0)+\vartheta_0\partial_\vartheta\gamma_t(\vartheta_0) & 0
\end{pmatrix}.
\end{align*}
  The family $\bigl(Z_t(\vartheta_0)\bigr)_{\vartheta_0\in\bar\Theta}$ satisfies the assumptions of Proposition \ref{prop:aou-propositions} since the limiting drift matrix is Hurwitz under Assumption \ref{ass:param_space} and the time-varying coefficients converge exponentially fast and uniformly by Remark \ref{rem:on_filter} \eqref{num:gamma_conv}. Hence, with $A(\vartheta_0)\coloneq\lim_{t\to\infty}A(t,\vartheta_0)$ and $B(\vartheta_0)\coloneq\lim_{t\to\infty}B(t,\vartheta_0)$, the limiting covariance matrix is given by
\begin{equation}\label{eq:Limiting_Covariance}
    \Sigma(\vartheta_0)
    =\int_0^\infty e^{A(\vartheta_0)s}B(\vartheta_0)B(\vartheta_0)^\top e^{A(\vartheta_0)^\top s}ds.
\end{equation}
Equivalently, $\Sigma(\vartheta_0)$ is the unique positive semidefinite solution of
\begin{equation}\label{eq:lyapunov-covariance}
  A(\vartheta_0)\Sigma(\vartheta_0)
  +\Sigma(\vartheta_0)A(\vartheta_0)^\top
  +B(\vartheta_0)B(\vartheta_0)^\top=0.
\end{equation}
  Likewise, define $\widetilde Z_t(\vartheta_0,\vartheta)\coloneq\bigl(X_t,m_t(\vartheta_0), m_t(\vartheta)\bigr)$. Then $\bigl(\widetilde Z_t(\vartheta_0,\vartheta)\bigr)_{\vartheta_0,\vartheta\in\Theta}$ satisfies the assumptions of Proposition \ref{prop:aou-propositions} by Assumption \ref{ass:param_space} and Remark \ref{rem:on_filter} \eqref{num:gamma_conv}.
\end{remark}
\begin{remark}
Asymptotic properties of filters are well established for the non-correlated case, see e.g. \citep{budhirajaAsymptoticStabilityErgodicity2003,bishop2017}. Proposition \ref{prop:aou-propositions} complements general concentration results for diffusions, e.g. by \citet*{kutoyants2004statistical,dalalyanAsymptoticStatisticalEquivalence2007,trottnerConcentrationAnalysisMultivariate2023,aeckerle-willemsConcentrationScalarErgodic2021}. We use the Gaussian structure for a direct proof including the form of the limiting invariant measure, and also obtain the exponential decay of the autocovariances.
\end{remark}

\subsection{The Fisher information}\label{sec:info-loss}

Having established the concentration of the joint process $\bigl(X_t,m_t(\vartheta_0),\partial_\vartheta m_t(\vartheta_0)\bigr)$, we proceed with deriving the score and the Fisher information of the statistical model. The score at $\vartheta_0$ is defined as the (mean square) derivative of the log-likelihood with respect to the parameter from Lemma \ref{lem:score-derivation} evaluated at $\vartheta_0$.
\begin{lemma}\label{lem:score-derivation}
  For any $\vartheta_0\in\Theta$ the log-likelihood $\vartheta\mapsto\ell_T(\vartheta)$ is differentiable in mean square at $\vartheta_0$ with respect to $\Prob_{\vartheta_0}$, that is
  there exists $S_T(\vartheta_0)\in L^2(\Prob_{\vartheta_0})$ (the score) such that
  \[
    \E_{\vartheta_0}\left[\left(
      \frac{\ell_T(\vartheta_0+h)-\ell_T(\vartheta_0)}{h} - S_T(\vartheta_0)
      \right)^2\right]\xrightarrow[h\to 0]{}0.
  \]
  Moreover, the mean square derivative is given by
  \begin{align*}
    S_T(\vartheta_0)
     & = \int_0^T \Bigl(m_t(\vartheta_0)+\vartheta_0\partial_\vartheta m_t(\vartheta_0) \Bigr)dX_t  - \int_0^T \Bigl(\vartheta_0 m_t(\vartheta_0)^2+\vartheta_0^2 m_t(\vartheta_0)\partial_\vartheta m_t(\vartheta_0) \Bigr)dt \\
     & \quad - \int_0^T \Bigl(aX_t m_t(\vartheta_0)+a\vartheta_0 X_t\partial_\vartheta m_t(\vartheta_0) \Bigr)dt.
  \end{align*}
\end{lemma}

\noindent The proof of Lemma \hyperref[prf:score-derivation]{\ref{lem:score-derivation}} can be found in Appendix \ref{prf:score-derivation}.

Using \eqref{eq:Xadapted}, it follows that under $\mathbb{P}_{\vartheta_0}$ the score admits the martingale representation
\[
  S_T(\vartheta_0)=\int_0^T\bigl(m_t(\vartheta_0)+\vartheta_0\partial_\vartheta m_t(\vartheta_0)\bigr)d\bar W_t.
\]
The Fisher information at $\vartheta_0$ is defined as the variance of the score under $\Prob_{\vartheta_0}^T$:
\begin{equation}
  \label{eq:fisher}
    I_{X,T}(\vartheta_0)\coloneqq I_T(\vartheta_0)
     \coloneqq \mathbb{E}_{\vartheta_0}\left[\int_0^T \bigl(m_t(\vartheta_0)+\vartheta_0 \partial_\vartheta m_t(\vartheta_0)\bigr)^2 dt\right],\quad \vartheta_0\in\bar{\Theta}, T>0.
\end{equation}

\begin{proposition}\label{prop:uniform-fisher}
With $v(\vartheta_0)=(0,0,1,\vartheta_0)^\top$ and $\Sigma(\vartheta_0)$ from \eqref{eq:Limiting_Covariance} define the limiting Fisher information 
\begin{equation}\label{eq:fisher-info-rate}
    I(\vartheta_0)\coloneqq v(\vartheta_0)^\top \Sigma(\vartheta_0)v(\vartheta_0)\in (0,\infty),\quad \vartheta_0\in \bar{\Theta}.
\end{equation}
Then for any compact subset $K\subset\Theta$, $0<\inf_{\vartheta_0\in K}I(\vartheta_0)\le\sup_{\vartheta_0\in K}I(\vartheta_0)<\infty$, and
\begin{equation*}
    \sup_{\vartheta_0\in K} \left|\frac{1}{T}I_T(\vartheta_0) - I(\vartheta_0)\right|\to 0,\quad \text{as } T\to\infty.
\end{equation*}
\end{proposition}
\noindent The proof of Proposition \hyperref[prf:uniform-fisher]{\ref{prop:uniform-fisher}} can be found in Appendix \ref{prf:uniform-fisher}.

We can quantify the information loss due to observing only $X^T$ instead of the full process $(X^T,Y^T)$. Let $\Sigma(\vartheta_0)$ denote the stationary covariance matrix corresponding to $Z_t(\vartheta_0)=\bigl(X_t,Y_t,m_t(\vartheta_0),\partial_\vartheta m_t(\vartheta_0)\bigr)^\top$, $t\ge 0$, from \eqref{eq:lyapunov-covariance}. Let $I_{X,T}(\vartheta_0)$ and $I_{X,Y,T}(\vartheta_0)$ denote the Fisher information when observing only $X^T$ and when observing both $X^T$ and $Y^T$, respectively, (see \eqref{eq:fisher-info-rate} and Appendix \ref{sec:information_ratio_conv}). Let $S_{X,Y,T}(\vartheta)\coloneq\partial_\vartheta \log L_{T,X,Y}(\vartheta)$ denote the full-data score. Under the assumption that differentiation and conditional expectation can be interchanged, the law of total variance yields
\[
  I_{X,Y,T}(\vartheta)
  =
  I_{X,T}(\vartheta)
  +
  \mathbb E_\vartheta\Bigl[
    \operatorname{Var}_\vartheta\bigl(S_{X,Y,T}(\vartheta)\mid\mathcal F_T^X\bigr)
  \Bigr],
\]
so that the difference between full and partial information is given by the conditional variance term. Denoting the limiting Fisher information rates by $I_X(\vartheta_0)\coloneq\lim_{T\to\infty}\frac{1}{T}I_{X,T}(\vartheta_0)$ and $I_{X,Y}(\vartheta_0)\coloneq\lim_{T\to\infty}\frac{1}{T}I_{X,Y,T}(\vartheta_0)$, we define the information ratio
\[
  R(\vartheta_0)\coloneq\frac{I_X(\vartheta_0)}{I_{X,Y}(\vartheta_0)},
\]
where small values of $R(\vartheta_0)$ correspond to a substantial information loss, whereas values of $R(\vartheta_0)$ close to one indicate only minor information loss.
We find the representations
\begin{equation}\label{eq:information_repr}
I_X(\vartheta_0) = \bigl[\Sigma(\vartheta_0)\bigr]_{3,3}
+2\vartheta_0\bigl[\Sigma(\vartheta_0)\bigr]_{3,4}
+\vartheta_0^{2}\bigl[\Sigma(\vartheta_0)\bigr]_{4,4}, \quad I_{X,Y}(\vartheta_0) = \bigl[\Sigma(\vartheta_0)\bigr]_{2,2},
\end{equation}
where the subscripts indicate the respective matrix entries.
Solving the Lyapunov equation in \eqref{eq:lyapunov-covariance} for the stationary covariance matrix yields explicit expressions for $I_X(\vartheta_0)$ and $I_{X,Y}(\vartheta_0)$ and we derive in Appendix \ref{sec:information_ratio_conv} that
\[
  R(\vartheta_0)\xrightarrow{|b|\to\infty}1,\qquad
  R(\vartheta_0)\xrightarrow{|\vartheta_0|\to\infty}0,\qquad
  R(\vartheta_0)\xrightarrow{c\to-\infty}0,\qquad
  R(\vartheta_0)\xrightarrow{a\to-\infty}\rho,
\]
where $\rho\in(0,1)$. In each limit, all other parameters are fixed.

Intuitively, as the influence of $X^T$ on $Y^T$ grows ($|b|\to\infty$), the latter is dominated by the term $bX_tdt$. Since $X^T$ and $b$ are known, observing $Y$ adds vanishing additional Fisher information and the information ratio converges to one. On the other hand, as $|\vartheta_0|\to\infty$, the drift of $X^T$ is dominated by $\vartheta_0 Y_t$ and changes in $\vartheta_0$ can be offset by opposing changes in $Y^T$ when only $X^T$ is observed. The information ratio thus converges to zero. 

In Figure \ref{fig:ratio-heatmap} (left) we compare the mean absolute errors of the MLE with full versus with partial information using Monte-Carlo estimates with $1000$ runs. The simulation is done using a Euler--Maruyama scheme with time stepping $\Delta t = 10^{-2}$. We can see that both estimators empirically achieve the theoretical convergence rate of $T^{-1/2}$. As expected, the MLE using full information outperforms the one using partial information. When $|b|$ becomes larger, its advantage decreases, aligning well with the previous analysis.

A heatmap of $R$ as a function of $b$ and $\vartheta_0$ is displayed in Figure \ref{fig:ratio-heatmap} (right) and shows that $R(\vartheta_0)$ approaches one as $|b|\to\infty$ and that $R(\vartheta_0)$ approaches zero as $|\vartheta_0|\to\infty$ with the remaining parameters fixed, matching the limits obtained analytically. 

\begin{figure}[t]
  \centering
  \begin{minipage}[c]{0.49\linewidth}
    \centering
    \includegraphics[width=\linewidth]{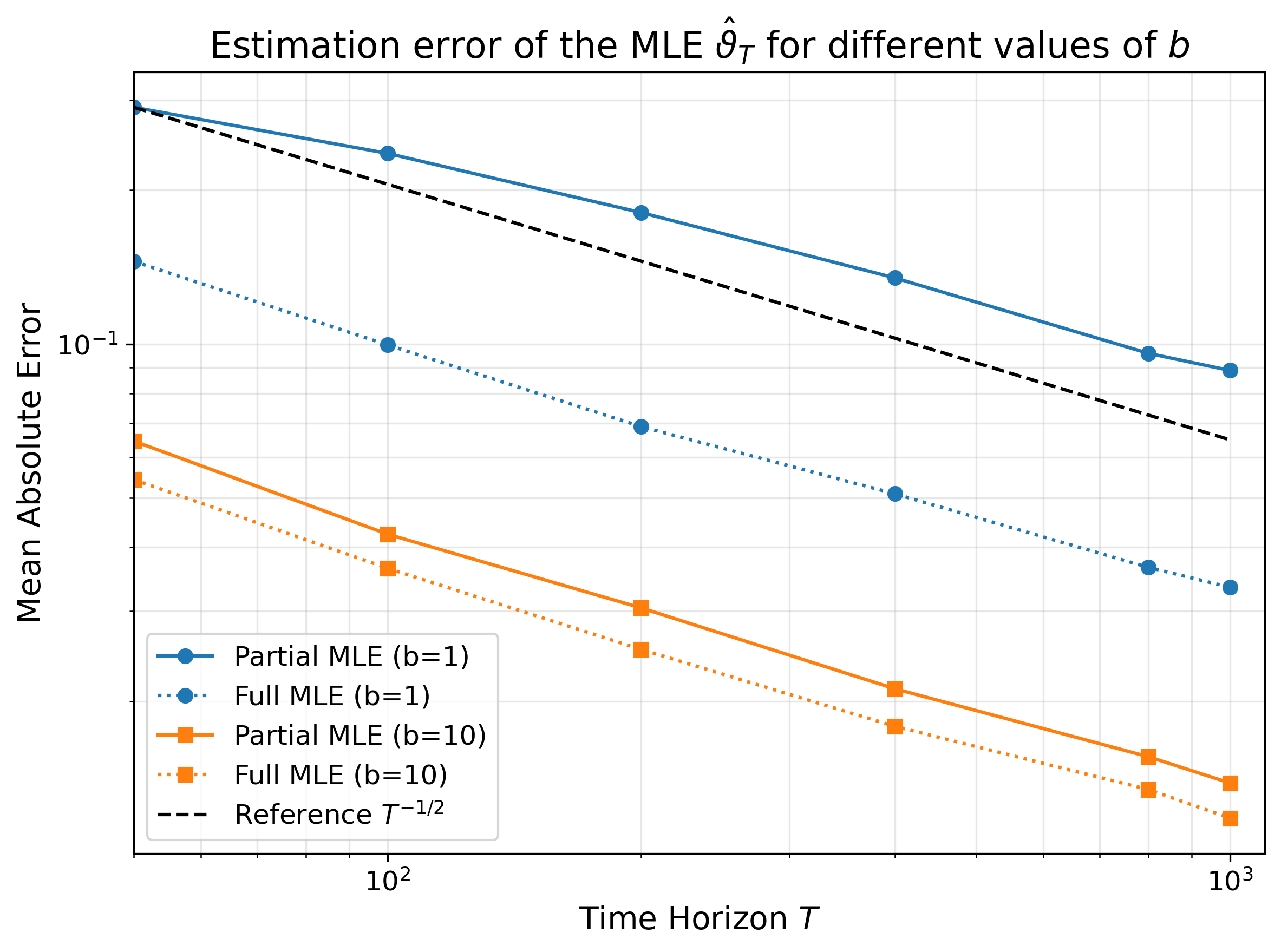}
  \end{minipage}\hfill
  \begin{minipage}[c]{0.49\linewidth}
    \centering
    \includegraphics[width=\linewidth]{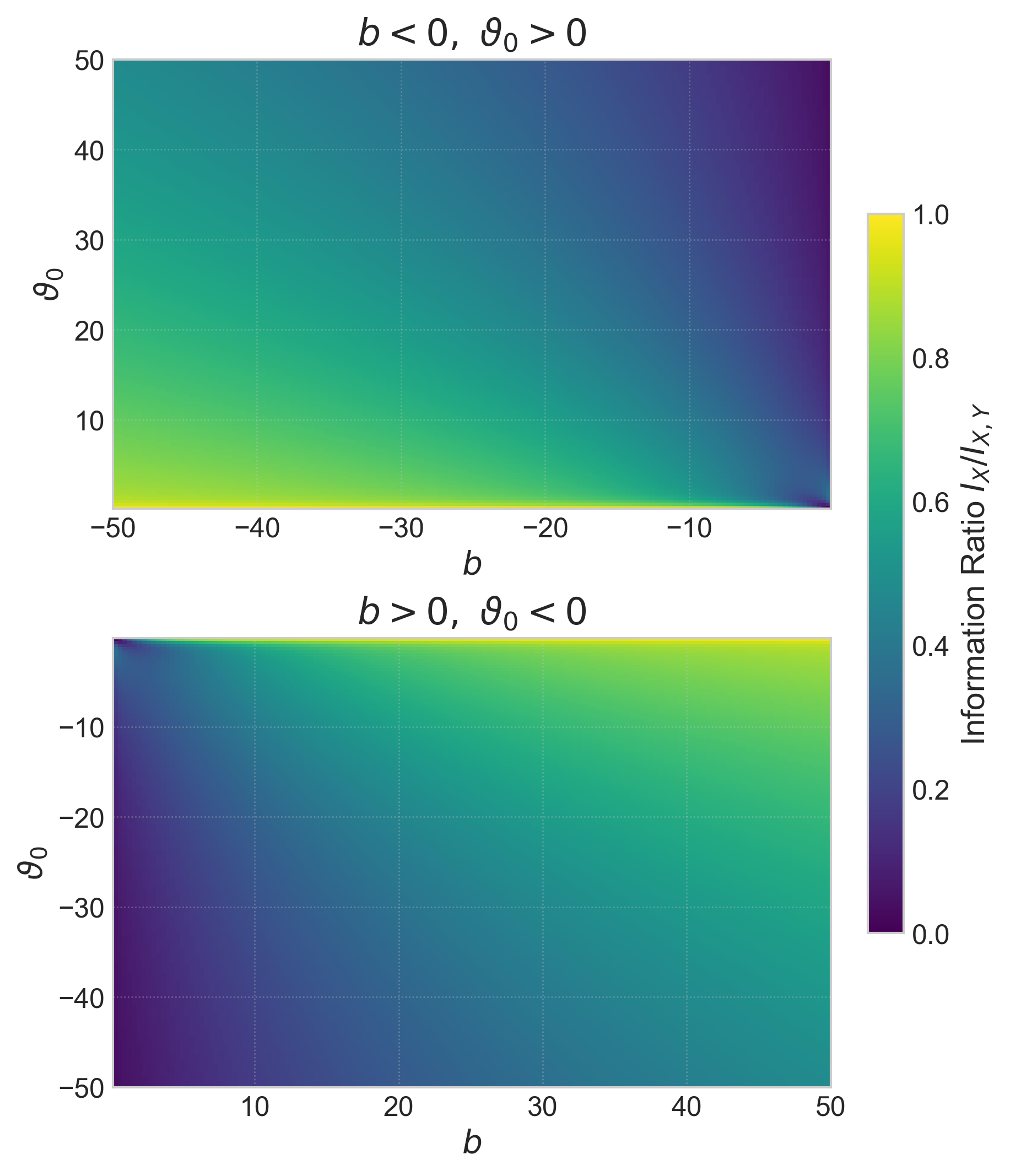}
  \end{minipage}
  \caption{Left: Comparison of the Monte-Carlo estimates of the mean absolute errors of the MLE given full versus partial information. Parameter set-up: $a=-1$, $c=-1$, $\vartheta_0=-0.5$ and $b\in\{1,10\}$ and $1000$ Monte-Carlo runs. Right: Information ratio $I_X/I_{X,Y}$ with $a=-1$, $c=-1$. }
  \label{fig:ratio-heatmap}
\end{figure}

\section{Main Results}\label{sec:main-results}
Having understood the long-time behaviour of the filters and the Fisher information, we proceed to derive the asymptotic behaviour of the MLE in our model. To this end, we first establish the LAN property for the statistical model and subsequently apply the likelihood ratio process framework of \citet{IbragimovHasminskii1981} to derive statistical guarantees for the MLE.
For $\vartheta_0 \in \Theta$ we consider the localised likelihood ratios given by
\begin{equation*}
  L_{T,\vartheta_0}(u) \coloneqq \frac{d\Prob_{\vartheta_0 + u I_T(\vartheta_0)^{-1/2}}}{d\Prob_{\vartheta_0}}(X^{T}),\quad u\in U_{T,\vartheta_0} \coloneq I_T(\vartheta_0)^{1/2}(\Theta-\vartheta_0).
\end{equation*}
We view these as elements from the space $C_0(\R)$ of continuous functions $f\colon\mathbb{R}\to\mathbb{R}$ vanishing at infinity, equipped with the supremum norm $\|f\|_\infty \coloneqq \sup_{u \in \R} |f(u)|$ by extending the domain from $U_{T,\vartheta_0}$ to $\R$ in such a way that their supremum-norms do not change.

\begin{theorem}
  \label{thm:lan_property}
  Let K be a compact subset of $\Theta$ and let $u\in U_{T,\vartheta_0}$ fixed. The log-likelihood ratio comparing the model with parameter $\vartheta_u=\vartheta_0+uI_T(\vartheta_0)^{-1/2}$ to the model with parameter $\vartheta_0$ can be expressed as
  \[
    \log L_{T,\vartheta_0}(u)
    = u\Delta_{T,\vartheta_0}-\tfrac12 u^2+\psi_T(u,\vartheta_0),
  \]
  where, uniformly in $\vartheta_0\in K$, $\Delta_{T,\vartheta_0} \xrightarrow[]d N(0,1)$ and $\psi_T(u,\vartheta_0)\xrightarrow{\Prob_{\vartheta_0}}0$ as $T\to\infty$, in the sense that
  for every bounded continuous function $h:\mathbb R\to\mathbb R$ we have
  \[
    \lim_{T\to\infty} \sup_{\vartheta_0\in K}
    \Bigl|
    \E_{\vartheta_0}\bigl[h(\Delta_{T,\vartheta_0})\bigr]
    -\E\bigl[h(\xi)\bigr]
    \Bigr|=0
  \]
  for $\xi \sim\mathcal N(0,1)$ and
  \[
    \lim_{T\to\infty} \sup_{\vartheta_0\in K}
    \Prob_{\vartheta_0}\big(|\psi_T(u,\vartheta_0)|>\varepsilon\big)=0.
  \]
  In particular, the LAN expansion from \citep[Definition II.2.2]{IbragimovHasminskii1981} is satisfied.
\end{theorem}
The proof can be found in Section \ref{sec:thm-proof}.

\begin{remark}\label{rem:lan}~
\begin{enumerate}[(i)]
\item The LAN property shows that our statistical model is asymptotically equivalent to a Gaussian location model, compare Section 7 of \citet{van-der-Vaart1996}. This equivalence is the main building block for analysing the asymptotic properties of the MLE.
  \item Similarly to Theorem \ref{thm:lan_property}, we can establish the uniform LAN property with the normalising factor $T^{-1/2}$ instead of $I_T(\vartheta_0)^{-1/2}$. For the perturbed parameter $\vartheta_u = \vartheta_0 + uT^{-1/2}$, the log-likelihood ratio relative to $\Prob_{\vartheta_0}$ admits the expansion 
  \[
    \log \tilde L_T(\vartheta_u;\vartheta_0)
    = u\Delta_{T,\vartheta_0} - \frac{1}{2} u^2 I(\vartheta_0) + \psi_T(u,\vartheta_0),\quad u\in T^{1/2}(\Theta-\vartheta_0),
  \]
  where $\Delta_{T,\vartheta_0} \xrightarrow[]{d} \mathcal{N}\bigl(0, I(\vartheta_0)\bigr)$ and $\psi_T(u,\vartheta_0) \xrightarrow{\Prob_{\vartheta_0}} 0$ as $T\to\infty$, uniformly over $\vartheta_0\in K$.
  \end{enumerate}
\end{remark}

We now state the main results for the MLE. 

\begin{theorem}
  \label{thm:main-mle}
  Let $K\subset\Theta$ be a compact set, then the MLE $\hat\vartheta_T$ satisfies, uniformly in $\vartheta_0\in K$, the following properties as $T\to\infty$:
  \begin{enumerate}[(i)]
    \item $\hat{\vartheta}_T$ is consistent.
    \item $\sqrt{I_T(\vartheta_0)}\bigl(\hat{\vartheta}_T-\vartheta_0\bigr)\xrightarrow[]{d} \mathcal{N}(0,1)$.
    \item All moments of $\sqrt{I_T(\vartheta_0)}\bigl(\hat{\vartheta}_T-\vartheta_0\bigr)$ converge to the corresponding moments of $\mathcal{N}(0,1)$.
  \end{enumerate}
\end{theorem}

The proof 
is given in Section \ref{sec:thm-proof}.

\begin{remark}
By combining the LAN property from Theorem \ref{thm:lan_property} and the asymptotic normality of the estimation error from Theorem \ref{thm:main-mle}, we obtain that the MLE $\hat{\vartheta}_T$ is asymptotically efficient in Fisher's sense according to Definition II.11.1 of \citet{IbragimovHasminskii1981}.
\end{remark}

\begin{remark}
    The following heuristic explains the idea behind the asymptotic normality in Theorem \ref{thm:main-mle}. 
We first establish that the likelihood ratio process $u\mapsto L_{T,\vartheta_0}(u)$ converges uniformly to a limiting likelihood ratio process
$u\mapsto L_{\vartheta_0}(u)$. This limiting process has an almost surely unique maximiser
$\hat u$, which is Gaussian. Therefore, one expects the normalised MLE to converge to the maximiser of the limiting likelihood. Formally, this idea can be written as
\[
\begin{aligned}
\Prob_{\vartheta_0}\left(I_T(\vartheta_0)^{1/2}\bigl(\argmax_{\vartheta\in\bar\Theta}L_T(\vartheta)-\vartheta_0\bigr) \cap A \neq \varnothing\right)
&= \Prob_{\vartheta_0}\left(\sup_{u\in A} L_{T,\vartheta_0}(u) \ge \sup_{u\notin A} L_{T,\vartheta_0}(u)\right) \\
&\xrightarrow[T\to\infty]{} \Prob_{\vartheta_0}\left(\sup_{u\in A} L_{\vartheta_0}(u) \ge \sup_{u\notin A} L_{\vartheta_0}(u)\right)\\
&= \Prob_{\vartheta_0}\left(\hat u\in A\right),
\end{aligned}
\]
for suitable continuity sets $A$. 
\end{remark}

\begin{remark}\label{rmk:conf-interval}
  Let $\alpha\in(0,1)$ and let $q_{1-\alpha/2}$ denote the $(1-\alpha/2)$-quantile of the standard normal distribution. Using the observed information evaluated at the MLE\[J_T(\hat{\vartheta}_T)=\int_0^T\bigl(m_t(\hat{\vartheta}_T)+\hat{\vartheta}_T \partial_{\vartheta}m_t(\hat{\vartheta}_T)\bigr)^2dt,\]
  we obtain for $\vartheta_0$ the confidence interval 
\begin{equation*}
  \mathcal I_T^{1-\alpha} = \bigl[ \hat\vartheta_T - q_{1-\alpha/2}J_T(\hat\vartheta_T)^{-1/2},
    \hat\vartheta_T + q_{1-\alpha/2}J_T(\hat\vartheta_T)^{-1/2}
  \bigr],
\end{equation*}
which has asymptotic coverage $1-\alpha$ as $T\to\infty$. This follows from the asymptotic normality statement in Theorem \ref{thm:main-mle}, Proposition \ref{prop:uniform-fisher}, Lemma \ref{lem:obs_inf_conv} and Slutsky's theorem, which imply
  \[
    \sqrt{J_T(\hat\vartheta_T)}
    \bigl(\hat\vartheta_T-\vartheta_0\bigr)
    \xrightarrow[]{d} \mathcal N(0,1).
  \]
\end{remark}

\begin{remark}\label{rmk:generalization}
One natural extension is to replace the linear dependence on the unknown
coupling parameter by a known regular reparametrisation. More precisely, suppose
that the coupling parameter is given by $\vartheta = f(\theta),  \theta\in\Theta$,
where $\Theta\subset\mathbb R$ is non-empty, open and bounded. We assume
that $f$ extends to a twice continuously differentiable function on an open neighbourhood of
$\bar\Theta$ and satisfies
\[
    \inf_{\theta\in\bar\Theta}|f'(\theta)|>0.
\]
In addition, the stability condition of Assumption \ref{ass:param_space} has to hold after reparametrisation, that is
\[
    a+c<0, \qquad b\neq 0, \qquad ac-bf(\theta)>0
    \quad \text{for all } \theta\in\bar\Theta .
\]

Under these assumptions, the preceding results transfer directly to the
parametrisation by $\theta$. Indeed, the likelihood is obtained from the
likelihood for $\vartheta$ by substituting $\vartheta=f(\theta)$ and the
regularity assumptions on $f$ allow the differentiability and LAN arguments to
be applied by the chain rule. The Fisher information is transformed according to
\[
    I_T^\theta(\theta_0)
    =
    f'(\theta_0)^2 I_T(f(\theta_0)).
\]
Since $\inf_{\theta\in\bar\Theta}|f'(\theta)|>0$ by assumption, the limiting Fisher
information remains positive whenever the original limiting Fisher information
for $\vartheta$ is positive. Consequently, the LAN property, consistency,
asymptotic normality, convergence of moments and asymptotic efficiency of the
MLE remain valid for the estimator of $\theta$.
\end{remark}

\subsection{Proofs of the main results} \label{sec:thm-proof}
 Before discussing the implications of the previous analysis for parameter estimation in linear SPDEs, we prove Theorems \ref{thm:lan_property} and \ref{thm:main-mle}. 
\begin{proof}[Proof of Theorem \ref{thm:lan_property}]
  The log-likelihood ratio comparing the model with parameter $\vartheta_u = \vartheta_0 + u I_T(\vartheta_0)^{-1/2}$, $u \in U_{T,\vartheta_0}$, to the model with parameter $\vartheta_0$ can be expressed as
  \begin{equation}\label{eq:lan-likelihood}
    \begin{aligned}
       & \log L_{T,\vartheta_0}(u)
      = \int_0^T \big( (\vartheta_um_t(\vartheta_u)
      - \vartheta_0m_t(\vartheta_0) \big) d\bar{W}_t
      - \frac{1}{2} \int_0^T \big( (\vartheta_um_t(\vartheta_u)
      - \vartheta_0m_t(\vartheta_0)\big)^{2} dt.
    \end{aligned}
  \end{equation}
  Note that this takes the form of a martingale minus one half of its quadratic variation. We begin by analysing the quadratic variation term and let $g_t(\vartheta) \coloneqq \vartheta m_t(\vartheta)$.
  By Lemma \ref{lem:mean-square-diff}, the function $\vartheta\mapsto g_t(\vartheta)$ is differentiable in mean square for any $t\ge 0$, such that we can write
  \begin{align}
  \label{eq:LAN-quadratic}
      Q_T(u) &\coloneq \frac{1}{2} \int_0^T \bigl(g_t(\vartheta_u) - g_t(\vartheta_0)\bigr)^2 dt\nonumber\\
      &= \frac{u^2}{2}\frac{\frac{1}{T}\int_0^T \Bigl(\frac{\partial g_t}{\partial\vartheta}(\vartheta_0)\Bigr)^{2} dt}{\frac{1}{T}I_T(\vartheta_0)}
     + \frac{1}{2}\int_0^T \Bigl(
    2u I_T(\vartheta_0)^{-1/2}\frac{\partial g_t}{\partial\vartheta}(\vartheta_0)
    R_t\bigl(\vartheta_u,\vartheta_0\bigr)
    + R_t^2\bigl(\vartheta_u,\vartheta_0\bigr)
    \Bigr)dt
  \end{align}
  for the remainder $R_t$ satisfying
\begin{align*}
\E_{\vartheta_0}\bigl[R_t^2(\vartheta_u,\vartheta_0)\bigr]
&=\E_{\vartheta_0}\Bigl[\Bigl(
(\vartheta_u-\vartheta_0)\bigl(m_t(\vartheta_u)-m_t(\vartheta_0)\bigr) \\
&\qquad \qquad
+\vartheta_0\bigl(
m_t(\vartheta_u)-m_t(\vartheta_0)
-(\vartheta_u-\vartheta_0)\partial_{\vartheta}m_t(\vartheta_0)
\bigr)
\Bigr)^2\Bigr] \\
&\le 2\bigl(C_1+|\Theta|^2C_2\bigr)(\vartheta_u-\vartheta_0)^4
\end{align*}
  uniformly in $T$ and $K$, where we used \eqref{eq:bound_rest} together with the moment estimates of Lemma \ref{lem:m_exp_bounds}.
  In particular, it follows that
  \begin{equation}
      \E_{\vartheta_0}\left[\int_0^T R_t^2\bigl(\vartheta_u,\vartheta_0\bigr)dt\right]
    \le \frac{Cu^4T}{I_T^2(\vartheta_0)}\to 0,\quad T\to\infty,\label{eq:Remaindervanishes}
  \end{equation}
  uniformly in $\vartheta_0\in K$. By Proposition \ref{prop:aou-propositions} \ref{item:aou-a} and Proposition \ref{prop:uniform-fisher}, the numerator and denominator in the first summand of \eqref{eq:LAN-quadratic} both converge uniformly to $I(\vartheta_0)$ and $\inf_{\vartheta_0 \in K}I(\vartheta_0)>0$. Hence, this summand converges in $L^1(\mathbb{P}_{\vartheta_0})$ to $u^2/2$, uniformly in $\vartheta_0\in K$. By the Cauchy-Schwarz inequality and the decay of the remainder by \eqref{eq:Remaindervanishes}, the second summand in \eqref{eq:LAN-quadratic} tends to zero in $L^1(\mathbb{P}_{\vartheta_0})$ as $T\to\infty$, uniformly in $\vartheta_0\in K$. Combining the convergence results for the two summands in \eqref{eq:LAN-quadratic}, we obtain that
  \begin{equation}
    \sup_{\vartheta_0\in K} \E_{\vartheta_0}\left[\left|Q_T(u)-\tfrac12 u^2\right|\right]\xrightarrow{T\to\infty }0,\label{eq:Convergence_QaudraticVariation}
  \end{equation}
  from which uniform convergence in probability of the quadratic variation $Q_T$ to $u^2/2$ follows.

Next we consider the martingale term in \eqref{eq:lan-likelihood}. By the same first-order expansion of $g_t$ at $\vartheta_0$ as above we obtain
\[
\int_0^T \bigl(g_t(\vartheta_u)-g_t(\vartheta_0)\bigr)d\bar W_t
=
u\Delta_{T,\vartheta_0}
+
\int_0^T R_t(\vartheta_u,\vartheta_0)d\bar W_t,
\]
where
\[
\Delta_{T,\vartheta_0}
\coloneq
I_T(\vartheta_0)^{-1/2}\int_0^T \frac{\partial g_t}{\partial\vartheta}(\vartheta_0)d\bar W_t.
\]
By \eqref{eq:Remaindervanishes}, the remainder term converges to zero in $L^2(\Prob_{\vartheta_0})$, hence in probability, uniformly over $\vartheta_0\in K$. Moreover, the (uniform) convergence in distribution of $\Delta_{T,\vartheta_0}$ to $\mathcal N(0,1)$ follows from Proposition 1.20 of \citet{kutoyants2004statistical}. Combining this with \eqref{eq:Convergence_QaudraticVariation}, we obtain the desired LAN expansion
\[
\log L_{T,\vartheta_0}(u)
=
u\Delta_{T,\vartheta_0}
-\tfrac12 u^2
+\psi_T(u,\vartheta_0).\qedhere
\]
\end{proof}

\begin{proof}[Proof of Theorem \ref{thm:main-mle}]
We verify that the statistical experiment $(\Prob_\vartheta)_{\vartheta\in\Theta}$ satisfies Conditions~N1--N4 of \citep[Chapter III]{IbragimovHasminskii1981} with the choice $\varphi(T,\vartheta)=I_T(\vartheta)^{-1/2}$.
In view of Proposition \ref{prop:uniform-fisher} and Theorem \ref{thm:lan_property}, it suffices to show the following properties:
\begin{enumerate}[(i)]
    \item\label{enum:Helper_MLE_Proof_C3} For any compact $K \subset \Theta$, some $0<\beta ,m <\infty$ and $K$-dependent constants $0<a,B<\infty$ it holds
\[ \sup_{\vartheta_0 \in K} \sup_{\substack{u,v \in U_{T,\vartheta_0} \\ |u| < R, |v| < R}} |u - v|^{-\beta}\mathbb{E}_{\vartheta_0} \left| L_{T,\vartheta_0}^{1/m}(u) - L_{T,\vartheta_0}^{1/m}(v) \right|^m < B(1+R^a),\quad R>0. \]
\item\label{enum:Helper_MLE_Proof_C4} For any compact $K \subset \Theta$ and any $N > 0$ there exists a time $0\le T_0<\infty$ depending on $N$ and $K$ such that
\[ \sup_{\vartheta_0 \in K} \sup_{T \ge T_0} \sup_{u \in U_{T,\vartheta_0}} |u|^N\mathbb{E}_{\vartheta_0} \bigl[L_{T,\vartheta_0}^{1/2}(u)\bigr] < \infty. \]
\end{enumerate}
Once these conditions have been established, Theorem \ref{thm:main-mle} follows from Theorem~III.1.1 of \citet{IbragimovHasminskii1981}.
We will first verify Condition  \eqref{enum:Helper_MLE_Proof_C3} with $\beta=m=2$.
To this end, define $\vartheta_1 \coloneqq \vartheta_0+I_T(\vartheta_0)^{-1/2}u$ and $\vartheta_2 \coloneqq \vartheta_0+I_T(\vartheta_0)^{-1/2}v$. Then the Radon Nikodym Theorem \citep[Corollary 7.34]{klenkeWahrscheinlichkeitstheorie2020} yields
  \begin{equation} \label{eq:girsanov}
    \mathbb E_{\vartheta_0} \left| L_{T,\vartheta_0}^{1/2}(u) - L_{T,\vartheta_0}^{1/2}(v) \right|^2 = 2- 2\mathbb E_{\vartheta_0}[L_{T,\vartheta_0}^{1/2}(u)L_{T,\vartheta_0}^{1/2}(v)] = 2- 2\mathbb E_{\vartheta_1}[L_{T}^{1/2}(\vartheta_2;\vartheta_1)],
  \end{equation} where we changed the measure from $\Prob_{\vartheta_0}$ to $\Prob_{\vartheta_1}$ and $L_{T}(\vartheta_2;\vartheta_1)$ is the likelihood ratio of parameters $\vartheta_1$ and $\vartheta_2$.
  An application by Ito's formula shows that the likelihood ratio from \eqref{eq:lan-likelihood} satisfies
  \[
    L_{T}^{1/2}(\vartheta_2;\vartheta_1)= 1+\int_0^T \frac{g_t(\vartheta_1,\vartheta_2)}{2} L_{t}^{1/2}(\vartheta_2;\vartheta_1)d\bar W_t-\int_0^T \frac{g_t^2(\vartheta_1,\vartheta_2)}{8} L_{t}^{1/2}(\vartheta_2;\vartheta_1)dt,
  \]
  where $g_t(\vartheta_1,\vartheta_2) \coloneqq (\vartheta_2m_t(\vartheta_2)-\vartheta_1m_t(\vartheta_1))$.
  We substitute the representation of $L_T^{1/2}(\vartheta_2;\vartheta_1)$ into \eqref{eq:girsanov} and take the expectation. The stochastic integral term vanishes in expectation since
$\mathbb{E}_{\vartheta_1}[g_t^2(\vartheta_1,\vartheta_2)L_t(\vartheta_2;\vartheta_1)]$
is uniformly bounded in $t\ge 0$. This follows by absorbing the likelihood term through a change of measure and then applying the moment bound on $g_t$ from Lemma~\ref{lem:m_exp_bounds}. By Tonelli's theorem and the basic inequality $2ab\le a^2+b^2$ we obtain
  \begin{align*}
      \mathbb{E}_{\vartheta_0}\bigl| L_{T,\vartheta_0}^{1/2}(u) - L_{T,\vartheta_0}^{1/2}(v) \bigr|^{2} &= \int_0^T \frac{1}{4}\mathbb{E}_{\vartheta_1}\bigl[L_{t}^{1/2}(\vartheta_2;\vartheta_1) g_t^2(\vartheta_1,\vartheta_2)\bigr] dt \\
      &\le \int_0^T\frac{1}{8}\left(\mathbb{E}_{\vartheta_1}[g_t^2(\vartheta_1,\vartheta_2)]+\mathbb{E}_{\vartheta_1}[g_t^2(\vartheta_1,\vartheta_2)L_{t}(\vartheta_2;\vartheta_1)] \right)dt\\
      &=\int_0^T \frac{1}{8}\Bigl( \mathbb{E}_{\vartheta_1}\bigl[g_t^{2}(\vartheta_1,\vartheta_2)\bigr] + \mathbb{E}_{\vartheta_2}\bigl[g_t^{2}(\vartheta_1,\vartheta_2)\bigr] \Bigr) dt \\
     & \le \frac{1}{4} T C (\vartheta_1-\vartheta_2)^{2}
    = C T \frac{(u-v)^{2}}{4 I_T(\vartheta_0)},
  \end{align*}
  where the uniform bound from Lemma \ref{lem:partial_m_bounds} and boundedness of the parameter space $\Theta$ has been used in the last inequality. Since $\frac{1}{T}I_T(\vartheta_0)\to I(\vartheta_0)\in (0,\infty)$ as $T\to\infty$ uniformly in $\vartheta_0\in K$ from Proposition \ref{prop:uniform-fisher}, we can find some $B\in(0,\infty)$ such that $CT/(4I_T(\vartheta_0))\le B$ for any $\vartheta_0\in K$, provided $T$ is sufficiently large. Condition \eqref{enum:Helper_MLE_Proof_C3} is verified.\\
  We continue with the verification of Condition \eqref{enum:Helper_MLE_Proof_C4}. We will show an even stronger statement, namely that for any compact subset $K \subset \Theta$, there exists some time $0\le T_0<\infty$ and some constant $0<C<\infty$ depending on $K$ such that
    \begin{equation}
        \sup_{\vartheta_0 \in K} \sup_{T \ge T_0} \sup_{u \in U_{T,\vartheta_0}} \mathbb{E}_{\vartheta_0} [L_{T,\vartheta_0}^{1/2}(u)] \le \exp(-Cu^2).\label{eq:Helper_MLE_Cond_C4}
    \end{equation}
To this end, define $\vartheta_u \coloneqq \vartheta_0+I_T(\vartheta_0)^{-1/2}u$ and let $1<p<2$.
  \begin{equation}
    \begin{aligned}
      \mathbb{E}_{\vartheta_0}\bigl[L_{T,\vartheta_0}^{1/2}(u)\bigr]
       & =\mathbb{E}_{\vartheta_0}\Bigl[\exp\Bigl(\frac12
                                    \int_0^T\bigl(g_t(\vartheta_u)-g_t(\vartheta_0)\bigr)d\bar{W_t}
                                    - \frac14
                                    \int_0^T\bigl(g_t(\vartheta_u)-g_t(\vartheta_0)\bigr)^{2}dt
                                    \Bigr)\Bigr]         \\
       & =\mathbb{E}_{\vartheta_0}\Bigl[\exp\Bigl(\frac12
                                    \int_0^T\bigl(g_t(\vartheta_u)-g_t(\vartheta_0)\bigr)d\bar{W_t}
                                    - \frac p8
                                    \int_0^T\bigl(g_t(\vartheta_u)-g_t(\vartheta_0)\bigr)^{2}dt
                                    \Bigr)\\
                                    &\quad\times\exp\Bigl(\Bigl(\frac p8-\frac14\Bigr)
                                    \int_0^T\bigl(g_t(\vartheta_u)-g_t(\vartheta_0)\bigr)^{2}dt
                                    \Bigr)\Bigr]         \\
       & \le\mathbb{E}_{\vartheta_0}\Bigl[\exp\Bigl(\frac p2
                                      \int_0^T\bigl(g_t(\vartheta_u)-g_t(\vartheta_0)\bigr)d\bar{W_t}
                                      - \frac{p^2}{8}
                                      \int_0^T\!\bigl(g_t(\vartheta_u)-g_t(\vartheta_0)\bigr)^{2}dt
                                      \Bigr)\Bigr]^{\frac 1p}       \\
       & \quad\times\mathbb{E}_{\vartheta_0}\Bigl[\exp\Bigl(\bigl(\frac p8-\frac14\bigr)q
                                                \int_0^T\bigl(g_t(\vartheta_u)-g_t(\vartheta_0)\bigr)^{2}dt
                                                \Bigr)\Bigr]^{\frac 1q} \\
       & \le\mathbb{E}_{\vartheta_0}\Bigl[\exp\Bigl(\bigl(\frac p8-\frac14\bigr)q
                                      \int_0^T\bigl(g_t(\vartheta_u)-g_t(\vartheta_0)\bigr)^{2}dt
                                      \Bigr)\Bigr]^{\frac 1q},
    \end{aligned}
  \end{equation}
  where we used Hölder inequality with $p$ and its conjugate $q=p/(p-1)$ in the first inequality. For the second inequality, we use the fact that $\exp\Bigl(\frac p2
  \int_0^T\bigl(g_t(\vartheta_u)-g_t(\vartheta_0)\bigr)d\bar{W_t}
  - \frac{p^2}{8}
  \int_0^T\bigl(g_t(\vartheta_u)-g_t(\vartheta_0)\bigr)^{2}dt
  \Bigr)$ is a local martingale. In particular, it is a supermartingale and its expectation is always bounded by 1. Furthermore, the choice of $1<p<2$ can be made such that $(\frac{1}{4}-\frac{p}{8})q \in (0,1)$. Therefore, we can use Jensen's inequality for concave functions to get
\begin{align*}
\mathbb{E}_{\vartheta_0}\Bigl[\exp\Bigl(\Bigl(\frac p8-\frac14\Bigr)q
      \int_0^T&\bigl(g_t(\vartheta_u)-g_t(\vartheta_0)\bigr)^{2}dt
      \Bigr)\Bigr]^{\frac 1q} \\
&\quad\le
\mathbb{E}_{\vartheta_0}\Bigl[\exp\Bigl(-
      \int_0^T\bigl(g_t(\vartheta_u)-g_t(\vartheta_0)\bigr)^{2}dt
      \Bigr)\Bigr]^{\frac14-\frac p8}.
\end{align*}
Note that $g_t(\vartheta_u)-g_t(\vartheta_0)$ is a mean square integrable zero mean Gaussian process since $g_t(\vartheta_0)=\vartheta_0 m_t(\vartheta_0)$ and $m_t(\vartheta_0), m_t(\vartheta_u)$ are jointly Gaussian with mean zero. Therefore, the assumptions of Lemma 3.1(ii) of \citet{KALLIANPUR1991284} are satisfied and we obtain
  \begin{equation}
    \mathbb{E}_{\vartheta_0}\bigl[L_{T,\vartheta_0}^{1/2}(u)\bigr] \le \exp\Big(-\kappa \frac{\mathbb{E}_{\vartheta_0}\bigl[\int_0^T(g_t(\vartheta_u)-g_t(\vartheta_0))^2dt\bigr]}{1+2\lVert K_{\vartheta_0,u}\rVert}\Big),\label{eq:MeanBound}
\end{equation}
  where $\lVert K_{\vartheta_0,u}\rVert$ is the operator norm of the covariance operator of the process $g_t(\vartheta_u)-g_t(\vartheta_0)$ and $\kappa =(1/4-p/8)$. 
  Consider first the numerator $\mathbb{E}_{\vartheta_0}[\int_0^T(g_t(\vartheta_u)-g_t(\vartheta_0))^2dt]$. We can use a Taylor expansion to find
  \begin{equation}\label{eq:lan-expansion}
    \begin{aligned}
      \mathbb{E}_{\vartheta_0}\Bigl[
                               \int_{0}^{T} &\bigl( g_t(\vartheta_u) - g_t(\vartheta_0) \bigr)^{2} dt
                                  \Bigr]
      = \mathbb{E}_{\vartheta_0}\Bigl[
                                    \int_{0}^{T} \Bigl(
                                    \frac{u}{\sqrt{I_T(\vartheta_0)}} \partial_{\vartheta} g_t(\vartheta_0)
                                    + R_t(\vartheta_u,\vartheta_0)
                                    \Bigr)^{2} dt
                                    \Bigr] \\[4pt]
       & = u^{2}
      + 2 \int_{0}^{T}
      \mathbb{E}_{\vartheta_0}\Bigl[
                                  \frac{u}{\sqrt{I_T(\vartheta_0)}}\partial_{\vartheta} g_t(\vartheta_0) R_t(\vartheta_u,\vartheta_0)
                                  \Bigr] dt
      + \int_{0}^{T}
      \mathbb{E}_{\vartheta_0}\bigl[
                                  R_t^{2}(\vartheta_u,\vartheta_0)
                                  \bigr] dt .
    \end{aligned}
  \end{equation}
  From Theorem \ref{thm:lan_property} we know that there exist constants $0<C_1,C_2<\infty$, such that
  \begin{equation}
    \Bigl|
     \int_0^T\! 2\mathbb{E}_{\vartheta_0}\!\Bigl[ \frac{u}{\sqrt{I_T(\vartheta_0)}} \partial_{\vartheta} g_t(\vartheta_0)R_t(\vartheta_u,\vartheta_0) \Bigr] dt
    + \int_0^T\! \mathbb{E}_{\vartheta_0}\bigl[ R_t^{2}(\vartheta_u,\vartheta_0) \bigr] dt
    \Bigr|
    \le \frac{C_1 T |u^3|}{I_T(\vartheta_0)^{\frac 32}} + \frac{C_2 u^{4} T}{I_T(\vartheta_0)^{2}}.
  \end{equation}
  Since $T^{-1} I_T(\vartheta_0)$ converges uniformly to $I(\vartheta_0)>0$ by Proposition \ref{prop:uniform-fisher}, there exists $0<\bar{T},C_{\min}<\infty$ such that for all $T \ge \bar T$ and all $\vartheta_0\in K$,
  \[
    \frac{T}{I_T(\vartheta_0)} \le \frac{2}{C_{\min}}.
  \]
  Hence, for $T \ge \bar T$,  we can bound
\begin{equation}\label{eq:bound_for_small_lambda}
  \begin{aligned}
    \Bigl|
    2 \int_0^T \mathbb{E}_{\vartheta_0}\biggl[
      \frac{u}{\sqrt{I_T(\vartheta_0)}}\partial_\vartheta g_t(\vartheta_0)
      R_t(\vartheta_u,\vartheta_0)
    \biggr] dt
    &+ \int_0^T \mathbb{E}_{\vartheta_0}\bigl[
      R_t^{2}(\vartheta_u,\vartheta_0)
    \bigr] dt
    \Bigr| \\
    & \le
    Cu^{2}\left(\frac{|u|}{\sqrt{I_T(\vartheta_0)}}+\frac{u^{2}}{I_T(\vartheta_0)}\right),
  \end{aligned}
\end{equation}
  where $C$ is a constant depending only on $C_1,C_2,C_{\min}$. Let $\lambda \coloneq \min\{1, 1/(4C)\}$, such that
  \begin{equation*}
      C\Bigl(\frac{|u|}{\sqrt{I_T(\vartheta_0)}}+\frac{u^2}{I_T(\vartheta_0)}\Bigr) \le \frac{1}{2},\quad \dfrac{|u|}{\sqrt{I_T(\vartheta_0)}} < \lambda.
  \end{equation*}
  Therefore, \eqref{eq:lan-expansion} and \eqref{eq:bound_for_small_lambda} yield that uniformly over $\vartheta_0 \in K$, $T\ge \bar T$ and for $|u|I_T(\vartheta_0)^{-1/2} < \lambda$, we obtain
  \begin{equation}\label{eq:bound_trace_small}
    \begin{aligned}
      \mathbb{E}_{\vartheta_0}\Bigl[\int_0^T \bigl(g_t(\vartheta_u)-g_t(\vartheta_0)\bigr)^{2} dt\Bigr]
       & \ge u^{2}
      -\Bigl|2\int_0^T \mathbb{E}_{\vartheta_0}\bigl[\frac{u}{\sqrt{I_T(\vartheta_0)}}R_t(\vartheta_u,\vartheta_0)\bigr] dt \\ 
      &\qquad \quad+\int_0^T \mathbb{E}_{\vartheta_0}\bigl[R_t^{2}(\vartheta_u,\vartheta_0)\bigr] dt \Bigr|\\
      &\ge \frac{u^{2}}{2}.
    \end{aligned}
  \end{equation}
  For $|u|I_T(\vartheta_0)^{-1/2} \ge \lambda$ we use the results from Appendix \ref{sec:ident}. From Proposition \ref{prop:uniform-separation} there exists a time $T_1(\lambda)$ such that for $t\ge T_1(\lambda)$, we obtain the lower bound
  \[
    \inf_{\vartheta_0\in K} \inf_{\substack{\vartheta_u\in \Theta\\ |\vartheta_u-\vartheta_0|\ge \lambda}}
   \mathbb{E}_{\vartheta_0}\left[(g_t(\vartheta_u)-g_t(\vartheta_0))^2\right]
    \ge K(\lambda) >0.
  \]  By Proposition \ref{prop:uniform-fisher}, the Fisher information $I(\vartheta_0)$ is bounded on $K$ and $T/I_T(\vartheta_0)\to 1/I(\vartheta_0)$ uniformly in $\vartheta_0\in K$.
  Since also $\frac{T_1}{I_T(\vartheta_0)}\to 0$ uniformly on $K$,
  there exists $T_2\ge T_1$ such that for all $T\ge T_2$ and all $\vartheta_0\in K$,
  $(T-T_1)/I_T(\vartheta_0)\ge 1/(2\max_{\vartheta_0\in K} I(\vartheta_0))$. Consequently, we obtain for all $T\ge T_2$ and $|\vartheta_u-\vartheta_0|\ge \lambda$, the bound
  \begin{equation}\label{eq:bound_trace_big}
    \begin{aligned}
      \mathbb{E}_{\vartheta_0}\Bigl[\int_0^T \bigl(g_t(\vartheta_u)-g_t(\vartheta_0)\bigr)^{2} dt\Bigr] \ge K(\lambda)(T-T_1) & \ge K(\lambda)(T-T_1)\frac{u^{2}}{I_T(\vartheta_0)|\Theta|^{2}} \\
 &\ge \frac{u^{2}K(\lambda)}{2\max_{\vartheta_0 \in K} I(\vartheta_0)|\Theta|^{2}}
      = Cu^{2}.
    \end{aligned}
  \end{equation}
  Combining Equations \eqref{eq:bound_trace_small} and \eqref{eq:bound_trace_big}, we deduce, for $T\ge T_0 \coloneqq\max(\bar T,T_2)$ and uniformly over all compact subsets of $\Theta$, that
  \begin{equation}\label{eq:bound_trace}
    \mathbb{E}_{\vartheta_0}\Bigl[\int_0^T \bigl(g_t(\vartheta_u)-g_t(\vartheta_0)\bigr)^{2} dt\Bigr] \ge \bar{C}u^2,
  \end{equation}
  for some constant $\bar{C}>0$ independent of $\vartheta_0\in K$ and $T$ sufficiently large. 
Consider next the operator norm $\lVert K_{\vartheta_0,u}\rVert$. The vector process $(X_t,g_t(\vartheta_u),g_t(\vartheta_0))$ under $\Prob_{\vartheta_0}$ satisfies the assumptions of Proposition \ref{prop:aou-propositions} (see Remark \ref{rem:stable-process}).
Then, Proposition \ref{prop:aou-propositions} \eqref{item:aou-c} yields exponential decay of the norm of the covariance matrix of the vector process and we obtain for the covariance under $\mathbb{P}_{\vartheta_0}$ the bound
\begin{align}
  \Bigl|\operatorname{Cov}_{\vartheta_0}\bigl(g_t(\vartheta_u)-g_t(\vartheta_0), g_s(\vartheta_u)-g_s(\vartheta_0)\bigr)\Bigr| 
  &\le 4\Bigl\|\operatorname{Cov}_{\vartheta_0}\Bigl( (X_t , g_t(\vartheta_u) ,g_t(\vartheta_0))^\top, \nonumber\\
  &\qquad\qquad\qquad\quad (X_s, g_s(\vartheta_u), g_s(\vartheta_0))^\top \Bigr)\Bigr\| \nonumber\\
  &\le 4Ce^{-\beta|t-s|}.\label{eq:cov_decay_bound}
\end{align}
  Combining Equation \eqref{eq:cov_decay_bound} and Schur's test \citep{Schur1911,halmos1978bounded}, we deduce 
  \begin{equation}\label{eq:bound-operator}
    \|K_{\vartheta_0,u}\|
    \le \sup_{t\in[0,T]} \int_{0}^{T} \bigl|\operatorname{Cov}_{\vartheta_0}\bigl(g_t(\vartheta_u)-g_t(\vartheta_0), g_s(\vartheta_u)-g_s(\vartheta_0)\bigr)\bigr| ds
    \le \frac{8C}{\beta}.
  \end{equation}
  \eqref{eq:MeanBound}, \eqref{eq:bound_trace} and \eqref{eq:bound-operator} yield a new constant $C$ that does not depend on T, $\vartheta_0$ or $u$ such that \eqref{eq:Helper_MLE_Cond_C4} holds for $T$ sufficiently large.
\end{proof}

\section{Perspectives}\label{sec:Perspectives}

An important instance of coupled systems with partial information is given by activator-inhibitor dynamics \citep[Chapter 2]{murrayMathematicalBiologyII2003}, which evolve both in time and in space. Here, local self-amplification by an activator species $X$ is counteracted by the lateral spread of a fast-diffusing inhibitor $Y$. The latter is typically not observable. Both $X$ and $Y$ are governed by (stochastic) PDEs. Motivated by the stochastic (non-linear) Meinhardt model introduced by \citet{Altmeyer2022}, it is promising to explore how the previous analysis can be extended to (linear) SPDEs. We refer the interested reader to \citep{Da-Prato_Zabczyk_2014} for an excellent exposition of SPDEs.

Let $\mathcal A$ be the Dirichlet Laplacian on $L^2(0,1)$ with eigenbasis $(e_k)_{k\in\mathbb{N}}$, eigenvalues $(-\lambda_k)_{k\in\mathbb{N}}$ and let $\mathbb{I}$ be the identity operator on $L^2(0,1)$. Consider the coupled linear system
\begin{align*}
  dX_t &= \bigl((\mathcal A+a\mathbb{I})X_t+\vartheta Y_t\bigr)dt+dW_t^X,\\
  dY_t &= \bigl(bX_t+(\mathcal A+c\mathbb{I})Y_t\bigr)dt+dW_t^Y,
\end{align*}
$t\in[0,T]$, where $W^X$ and $W^Y$ are independent cylindrical Brownian motions on $L^2(0,1)$.
Assume that $X^T$ is observed and $Y^T$ is hidden. Writing $X_t^k=\langle X_t,e_k\rangle_{L^2(0,1)}$, $Y_t^k=\langle Y_t,e_k\rangle_{L^2(0,1)}$ for $k\in\mathbb{N}$ and denoting the standard Brownian motions by $B_t^{X,k}\coloneq W_t^X(e_k)$, $B_t^{Y,k}\coloneq W_t^Y(e_k)$, we obtain for each $k\in\mathbb{N}$ the Fourier modes
\begin{align*}
  dX_t^k &= \bigl((a-\lambda_k)X_t^k+\vartheta Y_t^k\bigr)dt+dB_t^{X,k},\\
  dY_t^k &= \bigl(bX_t^k+(c-\lambda_k)Y_t^k\bigr)dt+dB_t^{Y,k}.
\end{align*}
Each Fourier mode has exactly the same form as \eqref{eq:XYsystem}, with $a$ and $c$ replaced by $a-\lambda_k$ and $c-\lambda_k$, respectively. Moreover, these projected systems are independent for different values of $k\in\mathbb{N}$, because the Laplacian is diagonal in the basis $(e_k)_{k\in \mathbb{N}}$ and the standard Brownian motions $(B_t^{X,k})_{t\in[0,T]}$ and $(B_t^{Y,k})_{t\in[0,T]}$ are independent across different Fourier modes.

As in the spectral approach in the parameter estimation for SPDEs (see, e.g., \citet{Huebner1995}), consider observing the first $n\in\mathbb{N}$ components $(X_t^k: t\in[0,T], k=1,\dots, n)$.
Given the stability and identifiability conditions from Assumption~\ref{ass:param_space}, the results of Sections~\ref{sec:setup} and~\ref{sec:main-results} apply for each individual mode $k=1,\dots, n$. Moreover, because the different Fourier modes are independent, the projected
log-likelihood and the Fisher information are additive. Theorem \ref{thm:main-mle} shows that the MLE $\hat\vartheta_T^{(n)}$ based on the first $n$ observed modes satisfies
\begin{equation*}
  \sqrt{T\sum_{k=1}^n I_k(\vartheta_0)}
  \bigl(\hat\vartheta_T^{(n)}-\vartheta_0\bigr)
  \xrightarrow[]{d}
  \mathcal N(0,1),
  \qquad T\to\infty,
\end{equation*}
where $I_k(\vartheta_0)$ denotes the Fisher information rate of the $k$-th Fourier mode. Similarly to the computations in Section \ref{sec:info-loss}, one can also quantify the information gained at each mode $k=1,\dots, n$. Since $a-\lambda_k$ and $c-\lambda_k$ are of order $k^2$ by the Weyl asymptotic, the corresponding Fisher information rate satisfies $I_k(\vartheta_0)=O(k^{-6})$ and $\sum_{k\in\mathbb{N}} I_k(\vartheta_0)< \infty$. This mirrors the fact that $\vartheta$ is not recoverable in finite time even with access to all modes \citep{Huebner1995}.
Note that in the previous analysis we considered the large time asymptotic for a fixed number of modes. To fully exploit the information inherent to all modes, first letting $n\to\infty$ and then $T\to\infty$ would be required and poses a promising avenue for future research.

\appendix

\section{Appendix}\label{sec:tech-proofs}
We begin by proving technical lemmas that establish the differentiability and boundedness of the Riccati equation, the moment bounds for the filtering process $m_t$ and the mean square differentiability of $m_t$.

\begin{lemma}
  \label{lem:gamma_deriv}
  For each $\vartheta \in \Theta$, $\gamma_t(\vartheta)$ has two derivatives
  with respect to $\vartheta$, which are bounded uniformly in
  $t\ge0$ and $\vartheta\in\bar\Theta$.
\end{lemma}

\begin{proof}
  Let $\gamma_t(\vartheta)$ denote the solution to the scalar Riccati ODE in
  \eqref{eq:cond_var}. Equation \eqref{eq:gamma_exp} shows that, for
  $\vartheta \neq 0$, $\gamma_t(\vartheta)$ is strictly increasing on
  $(0,\infty)$, remains positive for all $t>0$ and satisfies
  \[
    \gamma_t(\vartheta)\to
    \gamma_\infty(\vartheta)\coloneq
    \frac{c+\sqrt{c^{2}+\vartheta^{2}}}{\vartheta^{2}},
    \qquad t\to\infty .
  \]
  For $\vartheta=0$, \eqref{eq:cond_var} reduces to
  $d\gamma_t(0)/dt=2c\gamma_t(0)+1$ with $\gamma_0(0)=0$. Using that $c<0$ by Assumption \ref{ass:param_space} whenever $0 \in \Theta$, this ODE is solved by
  \[
    \gamma_t(0)=\frac{e^{2ct}-1}{2c}
    \xrightarrow{t\to\infty} -\frac{1}{2c}\eqqcolon\gamma_\infty(0)>0.
  \] Since
  $\gamma_\infty(\vartheta)\to -1/(2c)$ as $\vartheta\to0$, the map
  $\gamma_\infty$ extends continuously to any $\bar\Theta$ containing $0$.
  Hence $\gamma_\infty$ is strictly positive on $\bar\Theta$ and attains a
  finite maximum. Let
  $M_\gamma\coloneq\sup_{\vartheta\in\bar\Theta}\gamma_\infty(\vartheta)<\infty$.
  Monotonicity of $\gamma_t(\vartheta)$ and the bound on
  $\gamma_\infty(\vartheta)$ yield
  \begin{equation}\label{eq:bound_gamma}
    0<\gamma_t(\vartheta)\le \gamma_\infty(\vartheta)\le M_\gamma,
    \qquad t>0, \vartheta\in\Theta .
  \end{equation}
  We first derive a uniform bound for
  $v_t(\vartheta)\coloneqq\partial_\vartheta\gamma_t(\vartheta)$.
  By Theorem 2.11 of \citet{Teschl2012-ODE}, differentiating the Riccati equation from \eqref{eq:cond_var} with respect to $\vartheta$ yields
  \[
    \frac{dv_t(\vartheta)}{dt}
    =
    \bigl(2c-2\vartheta^2\gamma_t(\vartheta)\bigr)v_t(\vartheta)
    -2\vartheta\gamma_t(\vartheta)^2,
    \qquad v_0(\vartheta)=0.
  \]
  Solving this equation by variation of constants gives
  \[
    v_t(\vartheta)
    =
    \int_0^t
    \exp\Bigl(\int_s^t
    \bigl(2c-2\vartheta^2\gamma_\tau(\vartheta)\bigr)d\tau\Bigr)
    \bigl(-2\vartheta\gamma_s(\vartheta)^2\bigr)ds,\quad t\ge 0 .
  \]
  By Assumption \ref{ass:param_space} we have $(c,\vartheta)\neq(0,0)$ for $\vartheta\in \bar\Theta$, such that
  \[
    2c-2\vartheta^2\gamma_\infty(\vartheta)
    =
    -2\sqrt{c^2+\vartheta^2}
    \le -2\delta<0,
    \qquad
    \delta\coloneq\inf_{\vartheta\in\bar\Theta}\sqrt{c^2+\vartheta^2}>0.
  \]
  By \eqref{eq:gamma_exp}, $\gamma_t(\vartheta)$ converges to $\gamma_\infty(\vartheta)$ exponentially
  fast uniformly in $\vartheta\in\bar\Theta$. Consequently, there exists a time $0\le T_0<\infty$ such that
  \[
  2c-2\vartheta^2\gamma_t(\vartheta)\le -\delta,
  \qquad \vartheta\in\bar\Theta,\quad t\ge T_0 .
\]
On the finite interval $[0,T_0]$, the coefficient $2c-2\vartheta^2\gamma_t(\vartheta)$ is uniformly bounded from above. Hence, there exists a constant $0<C<\infty$, independent of $t$ and $\vartheta$, such that
\[
  \exp\Bigl(\int_s^t
  \bigl(2c-2\vartheta^2\gamma_\tau(\vartheta)\bigr)d\tau\Bigr)
  \le C e^{-\delta(t-s)},
  \qquad 0\le s\le t, \vartheta\in\bar\Theta .
\]
Combining the previous estimate, \eqref{eq:bound_gamma} and $\vartheta_{\max}\coloneq\sup_{\vartheta\in\bar\Theta}|\vartheta|$, we obtain
  \[
    |v_t(\vartheta)| \le 2\vartheta_{\max}M_\gamma^2C \int_0^t e^{-\delta(t-s)}ds\le \frac{2\vartheta_{\max}M_\gamma^2C}{\delta}.
  \]
We next derive a uniform bound for
  $w_t(\vartheta)\coloneqq\partial_\vartheta v_t(\vartheta)$. We have
  \[
    \frac{dw_t(\vartheta)}{dt}
    =
    \bigl(2c-2\vartheta^2\gamma_t(\vartheta)\bigr)w_t(\vartheta)
    + L_t(\vartheta),
    \quad w_0(\vartheta)=0,
  \]
where
\begin{equation*}
  L_t(\vartheta)
    =
    -2\gamma_t(\vartheta)^2
    -8\vartheta\gamma_t(\vartheta)v_t(\vartheta)
    -2\vartheta^2v_t(\vartheta)^2.    
\end{equation*}
Using the uniform bounds on $\gamma_t(\vartheta)$ and $v_t(\vartheta)$,
it follows that $L_t(\vartheta)$ is uniformly bounded on
$[0,\infty)\times\bar\Theta$. Applying the same variation-of-constants
estimate as in the first-derivative case yields a uniform bound for
$w_t(\vartheta)$ on $[0,\infty)\times\bar\Theta$.
\end{proof}

\begin{lemma}
  \label{lem:m_exp_bounds}
  There exists a constant $0<C<\infty$ such that for all
  $\vartheta_0,\vartheta\in\Theta$ and $t\ge0$,
  \[
    \mathbb{E}_{\vartheta_0}\bigl[m_t^2(\vartheta)\bigr] \le C,
    \qquad
    \mathbb{E}_{\vartheta_0}
    \bigl[(m_t(\vartheta+h)-m_t(\vartheta))^2\bigr] \le Ch^2 .
  \]
\end{lemma}

\begin{proof}
Under $\Prob_{\vartheta_0}$, expanding the $dX_t$ term in \eqref{eq:filter-eq} gives the SDE
  \[
    dm_t(\vartheta)
    =
    \Bigl(
    (c-\vartheta^2\gamma_t(\vartheta))m_t(\vartheta)
    +bX_t+\vartheta\vartheta_0\gamma_t(\vartheta)Y_t
    \Bigr)dt
    +\vartheta\gamma_t(\vartheta)dW_t^X .
  \]
  Applying It\^{o}'s formula to $f(x)=x^2$ and taking expectations gives
\begin{equation} \label{eq:ito-formula}
  \frac{d}{dt}\mathbb E_{\vartheta_0}\bigl[m_t^2(\vartheta)\bigr]
  =
  \mathbb E_{\vartheta_0}\bigl[
    2bX_tm_t(\vartheta)
    +2\vartheta\vartheta_0\gamma_t(\vartheta)m_t(\vartheta)Y_t
    +2(c-\vartheta^2\gamma_t(\vartheta))m_t^2(\vartheta)
    +\vartheta^2\gamma_t^2(\vartheta)
  \bigr].
\end{equation}
  Young's inequality yields
\begin{align*}
  2|b||\mathbb{E}_{\vartheta_0}[X_t m_t]| 
  &\le \varepsilon\mathbb{E}_{\vartheta_0}[m_t^2] + \frac{b^2}{\varepsilon}\mathbb{E}_{\vartheta_0}[X_t^2], \\
  2|\vartheta\vartheta_0|\gamma_t(\vartheta) |\mathbb{E}_{\vartheta_0}[Y_t m_t]| 
  &\le \varepsilon\mathbb{E}_{\vartheta_0}[m_t^2] + \frac{\vartheta^2\vartheta_0^2\gamma_t(\vartheta)^2}{\varepsilon} \mathbb{E}_{\vartheta_0}[Y_t^2]
\end{align*}
for any $\varepsilon>0$.
  Since $(X_t,Y_t)_{t\ge 0}$ is an Ornstein--Uhlenbeck process with Hurwitz drift
  matrix (see Assumption \ref{ass:param_space}) and zero initial condition, the second
  moments of $X_t$ and $Y_t$ are uniformly bounded on $t\ge0$ for each
  $\vartheta_0\in\bar\Theta$, compare \eqref{eq:Z-second-moment}. Since these moments are continuous in
  $\vartheta_0$, we obtain a uniform bound on $t\ge0$ and the compact set $\bar\Theta$.
  Moreover, $\gamma_t(\vartheta)$ is uniformly bounded in
  $t\ge0$ and $\vartheta\in\bar\Theta$ by Lemma \ref{lem:gamma_deriv}. Substituting these bounds into
  \eqref{eq:ito-formula}, we obtain a constant $0<C_\varepsilon<\infty$,
  independent of $t$, $\vartheta$ and $\vartheta_0$, such that
  \begin{equation}\label{eq:m-squared_bound}
    \frac{d}{dt}\mathbb E_{\vartheta_0}[m_t^2(\vartheta)]
    \le
    2(c-\vartheta^2\gamma_t(\vartheta)+\varepsilon)
    \mathbb E_{\vartheta_0}[m_t^2(\vartheta)]
    +C_\varepsilon .
  \end{equation}
By Remark \ref{rem:on_filter} \eqref{num:gamma_conv},
$c-\vartheta^2\gamma_t(\vartheta)$ converges to $-\sqrt{c^2+\vartheta^2}$ exponentially fast and uniformly on $\bar\Theta$. Since
$\sup_{\vartheta\in\bar\Theta}(-\sqrt{c^2+\vartheta^2})<0$, we can choose
$\lambda<0$, $\varepsilon>0$ and $T_0<\infty$ such that
\[
  2(c-\vartheta^2\gamma_t(\vartheta)+\varepsilon)\le\lambda,
  \qquad t\ge T_0,\vartheta\in\bar\Theta .
\]
Together with \eqref{eq:m-squared_bound}, Grönwall's inequality
\citep[Lemma 2.7]{Teschl2012-ODE} yields a
uniform bound for $\mathbb E_{\vartheta_0}[m_t^2(\vartheta)]$ on
$[T_0,\infty)$.
  
On $[0,T_0]$, the term $2(c-\vartheta^2\gamma_t(\vartheta)+\varepsilon)$
in \eqref{eq:m-squared_bound} is uniformly bounded from above in
$t\in[0,T_0]$ and $\vartheta\in\bar\Theta$. Thus Grönwall's inequality
gives a uniform bound on this interval as well. Combining both intervals, we
obtain for a constant $0<C<\infty$ the bound
\[\sup_{t\ge0, \vartheta,\vartheta_0\in\bar\Theta}
\mathbb{E}_{\vartheta_0}[m_t(\vartheta)^2]\le C.\]

  We next bound the second moment of
  $m_t(\vartheta+h)-m_t(\vartheta)$. Let
  $\delta_t(\vartheta)=m_t(\vartheta+h)-m_t(\vartheta)$. Expanding $dX_t$
  under $\Prob_{\vartheta_0}$ and applying Ito's formula yields
  \[
    \begin{aligned}
    \frac{d}{dt}\mathbb E_{\vartheta_0}[\delta_t^2(\vartheta)]
    =
    \mathbb E_{\vartheta_0}\Bigl[
      &2(c-\vartheta^2\gamma_t(\vartheta))\delta_t^2(\vartheta)
      +2(A_t(\vartheta+h)-A_t(\vartheta))m_t(\vartheta+h)\delta_t(\vartheta) \\
      &+2(C_t(\vartheta+h)-C_t(\vartheta))
      \vartheta_0Y_t\delta_t(\vartheta)
      +(C_t(\vartheta+h)-C_t(\vartheta))^2
    \Bigr],
    \end{aligned}
  \]
  with coefficients $A_t(\vartheta) \coloneq c - \vartheta^2 \gamma_t(\vartheta)$ and $C_t(\vartheta)\coloneq\vartheta \gamma_t(\vartheta)$. Since
  $\gamma_t(\vartheta)$ and its derivative with respect to $\vartheta$ are
  uniformly bounded by Lemma \ref{lem:gamma_deriv}, the functions $A_t(\vartheta)$
  and $C_t(\vartheta)$ have uniformly bounded first derivatives and are
  therefore Lipschitz-continuous with constants $L_A$ and $L_C$. Hence,
  Young's inequality yields
  \begin{equation*}
     \frac{d}{dt}\mathbb E_{\vartheta_0}[\delta_t^2(\vartheta)]
    \le
    2(c-\vartheta^2\gamma_t(\vartheta)+\varepsilon)
    \mathbb E_{\vartheta_0}[\delta_t^2(\vartheta)]
    \!+\frac{h^2L_A^2}{\varepsilon}
    \mathbb E_{\vartheta_0}[m_t^2(\vartheta)]
    +\frac{h^2L_C^2\vartheta_0^2}{\varepsilon}
    \mathbb E_{\vartheta_0}[Y_t^2]+\!h^2L_C^2.     
  \end{equation*}
  The coefficient
  $2(c-\vartheta^2\gamma_t(\vartheta)+\varepsilon)$ is the same as in
  \eqref{eq:m-squared_bound}. Using the uniform bounds on
  $\mathbb E_{\vartheta_0}[Y_t^2]$ and
  $\mathbb E_{\vartheta_0}[m_t^2(\vartheta)]$, the remaining terms are
  bounded by $Ch^2$, with $0<C<\infty$ independent of $t$, $\vartheta$,
  $\vartheta_0$ and $h$. Applying Grönwall's inequality on $[0,T_0]$ and
  on $[T_0,\infty)$ as above implies
  \[
    \mathbb E_{\vartheta_0}
    \bigl[(m_t(\vartheta+h)-m_t(\vartheta))^2\bigr]
    \le Ch^2 .\qedhere
  \]
\end{proof}


\begin{proof}[\textbf{Proof of Lemma \ref{lem:mean-square-diff}}]
\label{prf:m_twice_diff}
  Recall from \eqref{eq:filter-eq} that
  \begin{equation*}
    dm_t(\vartheta) = \bigl(A_t(\vartheta) m_t(\vartheta) + B_t(\vartheta) X_t\bigr) dt + C_t(\vartheta) dX_t
  \end{equation*}
  with coefficients
  \begin{align*}
    A_t(\vartheta) & \coloneqq c - \vartheta^2 \gamma_t(\vartheta), \quad
    B_t(\vartheta) \coloneqq b - \gamma_t(\vartheta) \vartheta a, \quad
    C_t(\vartheta) \coloneqq \vartheta \gamma_t(\vartheta).
  \end{align*}
  Lemma \ref{lem:gamma_deriv} shows that $A$, $B$ and $C$ have two bounded derivatives with respect to $\vartheta$. A second order Taylor expansion gives us
  \begin{align*}
    \frac{A_t(\vartheta+h) - A_t(\vartheta)}{h} - A_t'(\vartheta) & = A_t''(\tilde{\vartheta}_A)\frac{h}{2}, \\
    \frac{B_t(\vartheta+h) - B_t(\vartheta)}{h} - B_t'(\vartheta) & = B_t''(\tilde{\vartheta}_B)\frac{h}{2}, \\
    \frac{C_t(\vartheta+h) - C_t(\vartheta)}{h} - C_t'(\vartheta) & = C_t''(\tilde{\vartheta}_C)\frac{h}{2},
  \end{align*}
  where $\tilde{\vartheta}_A,\tilde{\vartheta}_B,\tilde{\vartheta}_C \in (\vartheta,\vartheta+h)$ and we denote by $Z'$ and $Z''$ the (partial) derivatives of $Z\in \{ A,B,C\}$ with respect to the parameter $\vartheta$.

  \noindent Let $\eta_t(\vartheta)$ be the candidate for the first derivative $\partial_\vartheta m_t(\vartheta)$ given by formally differentiating the equation for $m_t(\vartheta)$:
  \begin{equation}\label{eq:eq_derivative}
    d\eta_t(\vartheta) = \bigl(A_t'(\vartheta) m_t(\vartheta) + A_t(\vartheta) \eta_t(\vartheta) + B_t'(\vartheta) X_t\bigr) dt + C_t'(\vartheta) dX_t.
  \end{equation}
  Consider the process $  \Pi_t(\vartheta) \coloneqq \frac{1}{h}\bigl(m_t(\vartheta+h) - m_t(\vartheta)\bigr) - \eta_t(\vartheta).$ Substituting and rearranging terms, $\Pi_t(\vartheta)$ satisfies the SDE
  \begin{align*}
    d\Pi_t(\vartheta) = & \Bigl( A_t(\vartheta) \Pi_t(\vartheta) + \frac{h}{2} A_t''(\tilde{\vartheta}_A) m_t(\vartheta+h) \Bigr. \Bigl. + A_t'(\vartheta)\bigl(m_t(\vartheta+h)-m_t(\vartheta)\bigr)  \\&\quad + \frac{h}{2} B_t''(\tilde{\vartheta}_B) X_t \Bigr) dt
    + \frac{h}{2} C_t''(\tilde{\vartheta}_C) dX_t.
  \end{align*}
  Applying the variation of constants formula yields an explicit expression for $\Pi_t(\vartheta)$ given by
  \begin{align*}
    \Pi_t(\vartheta) & = \frac{h}{2} \int_0^t \exp\Bigl(\int_s^t A_r(\vartheta)dr\Bigr) A_s''(\tilde{\vartheta}_A) m_s(\vartheta+h) ds  \\
    &\quad+ \int_0^t \exp\Bigl(\int_s^t A_r(\vartheta)dr\Bigr) A_s'(\vartheta) \bigl(m_s(\vartheta+h)-m_s(\vartheta)\bigr) ds \\
                     &\quad+ \frac{h}{2} \int_0^t \exp\Bigl(\int_s^t A_r(\vartheta)dr\Bigr) B_s''(\tilde{\vartheta}_B) X_s ds + \frac{h}{2} \int_0^t \exp\Bigl(\int_s^t A_r(\vartheta)dr\Bigr) C_s''(\tilde{\vartheta}_C) dX_s.
  \end{align*}
After expanding $dX_s$ under $\Prob_{\vartheta_0}$ using \eqref{eq:Xadapted}, squaring $\Pi_t(\vartheta)$, applying $(a+b+c+d)^2\le 4(a^2+b^2+c^2+d^2)$ and taking expectations, we obtain
  \begin{align*}
    \mathbb{E}_{\vartheta_0}[\Pi_t^2(\vartheta)]
     \hspace{-1.1em}&\\
     &\le 4 \mathbb{E}_{\vartheta_0}\Bigl[\Bigl(\frac{h}{2} \int_0^t \exp\Bigl(\int_s^t A_r(\vartheta)dr\Bigr)\bigl(A_s''(\tilde{\vartheta}_A) m_s(\vartheta+h)+C_s''(\tilde{\vartheta}_C)\vartheta_0 m_s(\vartheta_0)\bigr) ds\Bigr)^2\Bigr] \\
     & \quad + 4 \mathbb{E}_{\vartheta_0}\Bigl[\Bigl(\int_0^t \exp\Bigl(\int_s^t A_r(\vartheta)dr\Bigr) A_s'(\vartheta)\bigl(m_s(\vartheta+h)-m_s(\vartheta)\bigr) ds\Bigr)^2\Bigr]                                                          \\
     & \quad + 4 \mathbb{E}_{\vartheta_0}\Bigl[\Bigl(\frac{h}{2} \int_0^t \exp\Bigl(\int_s^t A_r(\vartheta)dr\Bigr)\bigl(B_s''(\tilde{\vartheta}_B)+C_s''(\tilde{\vartheta}_C)a\bigr) X_s ds\Bigr)^2\Bigr]                                   \\
     & \quad + 4 \mathbb{E}_{\vartheta_0}\Bigl[\Bigl(\frac{h}{2} \int_0^t \exp\Bigl(\int_s^t A_r(\vartheta)dr\Bigr) C_s''(\tilde{\vartheta}_C) d\bar{W}_s\Bigr)^2\Bigr].
  \end{align*}
  Recall that $A$, $B$ and $C$ have two bounded derivatives with respect to $\vartheta$ by Lemma \ref{lem:gamma_deriv} and that $\mathbb{E}_{\vartheta_0}[X_t^2]$ is uniformly bounded in $t\ge 0$ by \eqref{eq:Z-second-moment}. In combination with the uniform moment and increment bounds for $m_t(\vartheta)$ from Lemma \ref{lem:m_exp_bounds} and the Itô isometry, we find
  \begin{align}
    \mathbb{E}_{\vartheta_0}[\Pi_t^2(\vartheta)] & \le Ch^2 \Bigl(\int_0^t \exp\Bigl(\int_s^t A_r(\vartheta)dr\Bigr) ds\Bigr)^2 \label{eq:helper_Differentiability}
  \end{align}
  for some $0<C<\infty$.
  Since $A_t(\vartheta)$ converges exponentially fast to $-\sqrt{c^2+\vartheta^2}<0$ by Remark \ref{rem:on_filter} \eqref{num:gamma_conv}, the integral in \eqref{eq:helper_Differentiability} 
  is uniformly bounded for all $t\ge 0$, and $\vartheta,\vartheta_0 \in \Theta$. Hence we find a new constant $0<C<\infty$ such that for all $t\ge 0$ and $\vartheta,\vartheta_0 \in \Theta$, the estimate 
  \[
    \mathbb{E}_{\vartheta_0}\bigl[\Pi_t^2(\vartheta)\bigr] \le Ch^2,
  \]
  holds. Letting $h\to0$ shows that $m_t(\vartheta)$ is mean square differentiable.
\end{proof}

\begin{lemma}
  \label{lem:partial_m_bounds}
  There exist constants $0<C_1,C_2<\infty$, independent of $t$, such that for
  all $\vartheta_0,\vartheta\in\Theta$ and $t\ge0$ the mean square
  derivative $\partial_\vartheta m_t(\vartheta)$ satisfies
  \[
    \mathbb{E}_{\vartheta_0}
    \Bigl[\bigl(\partial_\vartheta m_t(\vartheta)\bigr)^2\Bigr]\le C_1,
    \qquad
    \mathbb{E}_{\vartheta_0}
    \Bigl[
    \bigl(\partial_\vartheta m_t(\vartheta+h)
    -\partial_\vartheta m_t(\vartheta)\bigr)^2
    \Bigr]\le C_2h^2 .
  \]
\end{lemma}

\begin{proof}
From Lemma \ref{lem:mean-square-diff}, $\partial_\vartheta m_t(\vartheta)$ satisfies a linear SDE in which the term multiplying
$\partial_\vartheta m_t(\vartheta)$ is $c-\vartheta^2\gamma_t(\vartheta)$.
The remaining drift and diffusion coefficients are uniformly bounded and Lipschitz in
$\vartheta$ by Lemmas \ref{lem:gamma_deriv} and \ref{lem:m_exp_bounds}. Therefore, the claimed bounds follow
analogously to the proof of Lemma \ref{lem:m_exp_bounds}.
\end{proof}

\begin{proof}[\textbf{Proof of Proposition \ref{prop:aou-propositions}}]
  \phantomsection\label{prf:aou-prop}
We first prove $(\ref{item:aou-b})$ and $(\ref{item:aou-c})$ and then derive $(\ref{item:aou-a})$.
  
  \noindent\textit{Proof of (\ref{item:aou-b}).}
Fix $(\vartheta_0,\vartheta)\in\bar\Theta^{2}$. Define $D_t(\vartheta_0,\vartheta)\coloneq Z_t(\vartheta_0,\vartheta)-\bar Z_t(\vartheta_0,\vartheta)$ and abbreviate
\[
  \Delta A_t(\vartheta_0,\vartheta)\coloneq A(t,\vartheta_0,\vartheta)-A(\vartheta_0,\vartheta),\qquad
  \Delta B_t(\vartheta_0,\vartheta)\coloneq B(t,\vartheta_0,\vartheta)-B(\vartheta_0,\vartheta),\quad t\ge 0.
\]
  The process $D_t(\vartheta_0,\vartheta)$ satisfies the SDE
  \begin{equation}\label{eq:D-dynamics}
     dD_t(\vartheta_0,\vartheta) = A(\vartheta_0,\vartheta)D_t(\vartheta_0,\vartheta) dt
    + \Delta A_t(\vartheta_0,\vartheta)Z_t(\vartheta_0,\vartheta) dt
    + \Delta B_t(\vartheta_0,\vartheta) d\mathcal{W}_t .
  \end{equation}
  Let $\Phi_{\vartheta_0,\vartheta}(t,s)$ be the principal fundamental matrix of the linear system $dx/dt = A(t,\vartheta_0,\vartheta)x$, i.e.
  \[
    \frac{d}{dt} \Phi_{\vartheta_0,\vartheta}(t,s)=A(t,\vartheta_0,\vartheta)\Phi_{\vartheta_0,\vartheta}(t,s),
    \qquad \Phi_{\vartheta_0,\vartheta}(s,s)=\mathbb{I},
  \]
  where $\mathbb{I}$ is the identity matrix. The variation-of-constants formula gives
  \begin{equation}\label{eq:prin_mat-2par}
    \Phi_{\vartheta_0,\vartheta}(t,s)
    = e^{A(\vartheta_0,\vartheta)(t-s)}
    + \int_s^{t} e^{A(\vartheta_0,\vartheta)(t-u)}
    \big(A(u,\vartheta_0,\vartheta)-A(\vartheta_0,\vartheta)\big)
    \Phi_{\vartheta_0,\vartheta}(u,s)du
  \end{equation}
  for all $t\ge s\ge0$.
  After taking norms, the bounds from the assumptions imply
  \[
    \|e^{(t-u)A(\vartheta_0,\vartheta)}\|\le C e^{-\beta(t-u)},\qquad
    \|A(u,\vartheta_0,\vartheta)-A(\vartheta_0,\vartheta)\|\le M e^{-\alpha u}.
  \]
  Hence, we obtain
  \begin{equation}\label{eq:prin_mat2-2par}
    \|\Phi_{\vartheta_0,\vartheta}(t,s)\|
    \le C e^{-\beta(t-s)}
    + \int_s^{t} C e^{-\beta(t-u)}M e^{-\alpha u}\|\Phi_{\vartheta_0,\vartheta}(u,s)\|du .
  \end{equation}
  Set $\psi_{\vartheta_0,\vartheta}(t,s)\coloneq e^{\beta(t-s)}\|\Phi_{\vartheta_0,\vartheta}(t,s)\|$. Multiplying
  \eqref{eq:prin_mat2-2par} by $e^{\beta(t-s)}$ and applying Grönwall's inequality shows
  \[
    \psi_{\vartheta_0,\vartheta}(t,s) \le C \exp\big(\tfrac{CM}{\alpha}\big), \quad t\ge s\ge0.
  \]
  Therefore, for $K \coloneq C \exp\Big(\frac{CM}{\alpha}\Big)$, we obtain the bound
  \begin{equation}\label{eq:matrix_bound-2par}
    \|\Phi_{\vartheta_0,\vartheta}(t,s)\| \le Ke^{-\beta(t-s)},\quad  t\ge s\ge0.
  \end{equation}
  Since $Z_0(\vartheta_0,\vartheta)=0$ we can write
\begin{equation}\label{eq:mild-solution}
  Z_t(\vartheta_0,\vartheta)
  =\int_0^t \Phi_{\vartheta_0,\vartheta}(t,s)B(s,\vartheta_0,\vartheta) d\mathcal{W}_s .
\end{equation}
By the uniform boundedness assumption on $B$  and the bound in
\eqref{eq:matrix_bound-2par}, Itô's isometry gives, uniformly over $t\ge0$ and
$(\vartheta_0,\vartheta)\in\bar\Theta^2$, the estimate
\begin{equation} \label{eq:Z-second-moment}
  \begin{aligned}
    \mathbb E_{\vartheta_0}\bigl[\|Z_t(\vartheta_0,\vartheta)\|^2\bigr]
    &\le d\int_0^t \|\Phi_{\vartheta_0,\vartheta}(t,s)\|^2 \|B(s,\vartheta_0,\vartheta)\|^2 ds \\
    &\le dK^2\bar B^2\int_0^t e^{-2\beta(t-s)} ds \le \frac{dK^2\bar B^2}{2\beta} \eqqcolon M_Z.
  \end{aligned}
\end{equation}
  The solution for the limiting SDE is given by
  \[
    \bar Z_t(\vartheta_0,\vartheta)=\int_0^t e^{(t-s)A(\vartheta_0,\vartheta)}B(\vartheta_0,\vartheta) d\mathcal{W}_s.
  \]
  Hence, the same computation as for $Z_t$ yields the bound
  \begin{equation}\label{eq:barZ-second-moment}
    \sup_{(\vartheta_0,\vartheta)\in \bar \Theta^2}\sup_{t\ge0}\mathbb E_{\vartheta_0}\|\bar Z_t(\vartheta_0,\vartheta)\|^2 \le  \frac{dC^2 \bar B^2}{2\beta} \eqqcolon  M_{ \bar Z}.
  \end{equation}

Since $A(\vartheta_0,\vartheta)$ is Hurwitz, Theorem 4.6 of \citet{khalil2002nonlinear} ensures that for each $(\vartheta_0,\vartheta)$ there exists a unique positive definite matrix $P(\vartheta_0,\vartheta)$ satisfying
  \begin{equation}\label{eq:lyapunov-PQ}
    A(\vartheta_0,\vartheta)^\top P(\vartheta_0,\vartheta)
    +P(\vartheta_0,\vartheta)A(\vartheta_0,\vartheta)
    =-\mathbb{I}.
  \end{equation}
By the continuity of $A(\vartheta_0,\vartheta)$ and the uniqueness of $P(\vartheta_0,\vartheta)$, the map $(\vartheta_0,\vartheta)\mapsto P(\vartheta_0,\vartheta)$ is continuous. Hence, by compactness of $\bar{\Theta}^{2}$, the eigenvalues satisfy the uniform bounds
\[
  0 < p_{\min} \le \lambda_{\min}\bigl(P(\vartheta_0,\vartheta)\bigr)
  \le \lambda_{\max}\bigl(P(\vartheta_0,\vartheta)\bigr) \le p_{\max}.
\]
  Let $V_t(\vartheta_0,\vartheta)\coloneq D_t(\vartheta_0,\vartheta)^\top P(\vartheta_0,\vartheta)D_t(\vartheta_0,\vartheta)$. Applying Itô's formula to $V_t(\vartheta_0,\vartheta)$, substituting the dynamics in \eqref{eq:D-dynamics} and taking the expectation yields
\begin{align}\begin{split}
    \frac{d}{dt}\mathbb{E}_{\vartheta_0}\big[V_t(\vartheta_0,\vartheta)\big]
     & = \mathbb{E}_{\vartheta_0}\Big[
         D_t(\vartheta_0,\vartheta)^\top
         \big(A(\vartheta_0,\vartheta)^\top P(\vartheta_0,\vartheta)
         + P(\vartheta_0,\vartheta)A(\vartheta_0,\vartheta)\big)
         D_t(\vartheta_0,\vartheta)\Big] \\
     & \quad + \mathbb{E}_{\vartheta_0}\Big[2
             D_t(\vartheta_0,\vartheta)^\top
             P(\vartheta_0,\vartheta)
             \Delta A_t(\vartheta_0,\vartheta)
             Z_t(\vartheta_0,\vartheta)\Big] \\
     & \quad + \mathbb{E}_{\vartheta_0}\Big[\operatorname{tr}\big(
             P(\vartheta_0,\vartheta)
             \Delta B_t(\vartheta_0,\vartheta)
             \Delta B_t(\vartheta_0,\vartheta)^\top
             \big)\Big].
  \end{split}
  \label{eq:lyapunov-exp}
\end{align}
  Applying Young's inequality with $\varepsilon>0$, using the uniform bound on $\mathbb E_{\vartheta_0}[\|Z_t(\vartheta_0,\vartheta)\|^2]$ from \eqref{eq:Z-second-moment} and the exponential bounds
  $\|\Delta A_t(\vartheta_0,\vartheta)\| \le M e^{-\alpha t}$ and $\|\Delta B_t(\vartheta_0,\vartheta)\|\le M e^{-\alpha t}$, we deduce
\begin{equation*}
  \begin{aligned}
    &\mathbb{E}_{\vartheta_0}\Big[2 D_t(\vartheta_0,\vartheta)^\top P(\vartheta_0,\vartheta) \Delta A_t(\vartheta_0,\vartheta) Z_t(\vartheta_0,\vartheta) + \operatorname{tr}\big(P(\vartheta_0,\vartheta) \Delta B_t(\vartheta_0,\vartheta) \Delta B_t(\vartheta_0,\vartheta)^\top\big)\Big] \\
    &\qquad \le \varepsilon\frac{\mathbb{E}_{\vartheta_0}\big[V_t(\vartheta_0,\vartheta)\big]}{p_{\min}} + C_1e^{-2\alpha t}.
  \end{aligned}
\end{equation*}
Using \eqref{eq:lyapunov-PQ}, the uniform positive definiteness of
$P(\vartheta_0,\vartheta)$ and the preceding
estimate, we obtain
\begin{equation}
  \frac{d}{dt}\mathbb{E}_{\vartheta_0}\bigl[V_t(\vartheta_0,\vartheta)\bigr]
  \le
  -\left(\frac{1}{p_{\max}}-\frac{\varepsilon}{p_{\min}}\right)
  \mathbb{E}_{\vartheta_0}\bigl[V_t(\vartheta_0,\vartheta)\bigr]
  +C_1e^{-2\alpha t}.
\end{equation}
  Choose $\varepsilon>0$ such that
$\lambda_0\coloneqq 1/p_{\max}-\varepsilon/p_{\min}>0$. For any $0<\lambda<\min\{\lambda_0,2\alpha\}$, Grönwall's inequality then yields, uniformly in $(\vartheta_0,\vartheta)\in\bar{\Theta}^{2}$, the estimate
\[
  \mathbb{E}_{\vartheta_0}\bigl[V_t(\vartheta_0,\vartheta)\bigr]
  \le C_1\int_0^t e^{-\lambda_0(t-s)}e^{-2\alpha s} ds
  \le C_2e^{-\lambda t}.
\]
Since $\|D_t(\vartheta_0,\vartheta)\|^2\le p_{\min}^{-1}V_t(\vartheta_0,\vartheta)$, we obtain
\begin{equation}\label{eq:D-L2-decay}
  \mathbb{E}_{\vartheta_0}\bigl[\|D_t(\vartheta_0,\vartheta)\|^2\bigr]
  \le C_3e^{-\lambda t},
\end{equation}
 uniformly in $t\ge0$ and $(\vartheta_0,\vartheta)\in\bar\Theta^{2}$. Define
  \[
    \bar\Sigma(t,\vartheta_0,\vartheta)
    \coloneq\mathbb{E}_{\vartheta_0}\big[\bar Z_t(\vartheta_0,\vartheta)\bar Z_t(\vartheta_0,\vartheta)^\top\big].
  \]
Since $\bar Z_0(\vartheta_0,\vartheta)=0$ we have the representations
  \begin{equation}\label{eq:covariance-representation}
  \begin{aligned}
      \bar\Sigma(t,\vartheta_0,\vartheta)
    &=\int_0^t e^{sA(\vartheta_0,\vartheta)}B(\vartheta_0,\vartheta)B(\vartheta_0,\vartheta)^\top
    e^{sA(\vartheta_0,\vartheta)^\top}ds,\quad t\ge 0, \\
    \Sigma(\vartheta_0,\vartheta) &=\int_0^\infty e^{sA(\vartheta_0,\vartheta)}B(\vartheta_0,\vartheta)B(\vartheta_0,\vartheta)^\top
    e^{sA(\vartheta_0,\vartheta)^\top}ds.    
  \end{aligned}
  \end{equation}
  Using the assumptions $\|e^{sA(\vartheta_0,\vartheta)}\|\le C e^{-\beta s}$ and $\|B(\vartheta_0,\vartheta)\|\le \bar B$, we obtain
  \begin{equation}\label{eq:sigma_error}
    \big\|\Sigma(\vartheta_0,\vartheta)-\bar\Sigma(t,\vartheta_0,\vartheta)\big\|
    \le \int_t^\infty C^2 \bar B^{2} e^{-2\beta s}ds
    = \frac{C^2 \bar B^{2}}{2\beta}e^{-2\beta t}.
  \end{equation}
  We have the decomposition
  \begin{equation}\label{eq:decomposition-sigma}
    \Sigma(t,\vartheta_0,\vartheta)-\Sigma(\vartheta_0,\vartheta)
    =\Big(\Sigma(t,\vartheta_0,\vartheta)-\bar\Sigma(t,\vartheta_0,\vartheta)\Big)
    +\Big(\bar\Sigma(t,\vartheta_0,\vartheta)-\Sigma(\vartheta_0,\vartheta)\Big).
  \end{equation}
  Let $D_t(\vartheta_0,\vartheta)\coloneq Z_t(\vartheta_0,\vartheta)-\bar Z_t(\vartheta_0,\vartheta)$. Adding and subtracting the cross term $\bar Z_t Z_t^\top$ yields
  \begin{align*}
    \Sigma(t,\vartheta_0,\vartheta)-\bar\Sigma(t,\vartheta_0,\vartheta)
     & = \mathbb{E}_{\vartheta_0}\Big[D_t(\vartheta_0,\vartheta)D_t(\vartheta_0,\vartheta)^\top\Big]
    + \mathbb{E}_{\vartheta_0}\Big[D_t(\vartheta_0,\vartheta)\bar Z_t(\vartheta_0,\vartheta)^\top\Big] \\ &\quad+ \mathbb{E}_{\vartheta_0}\Big[\bar Z_t(\vartheta_0,\vartheta)D_t(\vartheta_0,\vartheta)^\top\Big].
  \end{align*}
  Applying the triangle inequality, submultiplicativity of the norm and the Cauchy--Schwarz inequality yields
\begin{equation}\label{eq:sigma_error2}
  \begin{aligned}
    \big\|\Sigma(t,\vartheta_0,\vartheta)-\bar\Sigma(t,\vartheta_0,\vartheta)\big\|
    &\le \mathbb{E}_{\vartheta_0}\big[\|D_t(\vartheta_0,\vartheta)\|^2\big] \\
    &\quad + 2\Big(\mathbb{E}_{\vartheta_0}\big[\|D_t(\vartheta_0,\vartheta)\|^2\big]\Big)^{1/2}
    \Big(\mathbb{E}_{\vartheta_0}\big[\|\bar Z_t(\vartheta_0,\vartheta)\|^2\big]\Big)^{1/2}.
  \end{aligned}
\end{equation}
  Combining the bounds from \eqref{eq:sigma_error} and \eqref{eq:sigma_error2} with the decomposition in \eqref{eq:decomposition-sigma} and the bounds in \eqref{eq:D-L2-decay} and \eqref{eq:barZ-second-moment} yields
  \[
    \sup_{(\vartheta_0,\vartheta)\in\bar{\Theta}^{2}}
    \big\|\Sigma(t,\vartheta_0,\vartheta)-\Sigma(\vartheta_0,\vartheta)\big\|
    \le C e^{-\lambda t/2} + \frac{C^2 \bar B^{2}}{2\beta}e^{-2\beta t}
    \le C'e^{-\gamma t},
  \] where $\gamma\coloneq\min\{\lambda/2,2\beta\}$.

  We proceed by showing that $(\vartheta_0,\vartheta)\mapsto \Sigma(\vartheta_0,\vartheta)$ is continuous on $\bar{\Theta}^{2}$.
  To this end, let $\bar{\Theta}^2\supset(\vartheta_0^{(n)},\vartheta^{(n)})\to(\vartheta_0,\vartheta)$ in $\bar{\Theta}^{2}$ for $n\to\infty$.
  By the continuity of $(\vartheta_0,\vartheta)\mapsto A(\vartheta_0,\vartheta)$ and $(\vartheta_0,\vartheta)\mapsto B(\vartheta_0,\vartheta)$, we obtain for each $s\ge 0$ that
\begin{align*}
&e^{sA(\vartheta_0^{(n)},\vartheta^{(n)})}B(\vartheta_0^{(n)},\vartheta^{(n)})B(\vartheta_0^{(n)},\vartheta^{(n)})^\top e^{sA(\vartheta_0^{(n)},\vartheta^{(n)})^\top} \\
  &\quad \to e^{sA(\vartheta_0,\vartheta)}B(\vartheta_0,\vartheta)B(\vartheta_0,\vartheta)^\top e^{sA(\vartheta_0,\vartheta)^\top}.
\end{align*}
  Moreover, by assumption, for all $n$ and all $s\ge 0$,
  \[
    \Bigl\|e^{sA(\vartheta_0^{(n)},\vartheta^{(n)})}B(\vartheta_0^{(n)},\vartheta^{(n)})B(\vartheta_0^{(n)},\vartheta^{(n)})^\top
    e^{sA(\vartheta_0^{(n)},\vartheta^{(n)})^\top}\Bigr\|
    \le C^2\bar B^{2}e^{-2\beta s}.
  \]
  Consequently, the integrand in \eqref{eq:covariance-representation} is dominated by an integrable function on $[0,\infty)$ and the dominated convergence theorem yields the continuity of $(\vartheta_0,\vartheta)\mapsto
\Sigma(\vartheta_0,\vartheta)$  on $\bar{\Theta}^{2}$.

\noindent\textit{Proof of (\ref{item:aou-c}).}
Let $\Phi_{\vartheta_0,\vartheta}(t,s)$ again denote the principal fundamental
matrix of $dx/dt=A(t,\vartheta_0,\vartheta)x$. Using the representation
\eqref{eq:mild-solution} of $Z_t(\vartheta_0,\vartheta)$, the semigroup
property, Itô isometry and independence of Brownian increments give, for
$t\ge s$,
\[
  \operatorname{Cov}_{\vartheta_0}\bigl(Z_t(\vartheta_0,\vartheta),Z_s(\vartheta_0,\vartheta)\bigr)
  =
  \Phi_{\vartheta_0,\vartheta}(t,s)
  \operatorname{Cov}_{\vartheta_0}\bigl(Z_s(\vartheta_0,\vartheta),Z_s(\vartheta_0,\vartheta)\bigr).
\]
Hence, using \eqref{eq:matrix_bound-2par} and \eqref{eq:Z-second-moment}, we obtain 
\[
  \bigl\|\operatorname{Cov}_{\vartheta_0}\bigl(Z_t(\vartheta_0,\vartheta),Z_s(\vartheta_0,\vartheta)\bigr)\bigr\|
  \le K e^{-\beta(t-s)} M_Z,
\]
uniformly in
$(\vartheta_0,\vartheta)\in\bar\Theta^{2}$ and $t\ge s\ge0$.
Similarly, we find 
\[
  \operatorname{Cov}_{\vartheta_0}\bigl(Z_t(\vartheta_0,\vartheta),Z_s(\vartheta_0,\vartheta)\bigr)
  =
  \operatorname{Cov}_{\vartheta_0}\bigl(Z_t(\vartheta_0,\vartheta),Z_t(\vartheta_0,\vartheta)\bigr)
  \Phi_{\vartheta_0,\vartheta}(s,t)^\top,\quad s\ge t,
\]
and the same bound holds with $s-t$ in place of $t-s$. Combining both cases gives
\begin{equation*}
  \sup_{(\vartheta_0,\vartheta)\in\bar\Theta^{2}}
  \bigl\|\operatorname{Cov}_{\vartheta_0}\bigl(Z_t(\vartheta_0,\vartheta),Z_s(\vartheta_0,\vartheta)\bigr)\bigr\|
  \le C e^{-\beta|t-s|}
\end{equation*}
for all $t,s\ge0$.

\noindent\textit{Proof of (\ref{item:aou-a}).} Recall that $h(z)=z^\top \mathcal M z+\mathcal K^\top z + \mathcal C$ with a symmetric matrix $\mathcal M\in\mathbb{R}^{d\times d}$. Since $Z_t(\vartheta_0,\vartheta)$ is centred Gaussian, from \ref{prop:aou-propositions} \eqref{item:aou-b}, we obtain a constant $C$ depending on $\mathcal M$, such that
\begin{equation} \label{eq:bias}
\begin{aligned}
&\sup_{(\vartheta_0,\vartheta)\in\bar\Theta^2}
\left|
\frac1T\int_0^T \mathbb E_{\vartheta_0}\bigl[h(Z_t(\vartheta_0,\vartheta))\bigr]dt
-
\mathbb E_{\mu_{\vartheta_0,\vartheta}}\bigl[h(\bar Z(\vartheta_0,\vartheta))\bigr]
\right| \\
&\quad =
\sup_{(\vartheta_0,\vartheta)\in\bar\Theta^2}
\left|
\frac1T\int_0^T
\operatorname{tr}\bigl(\mathcal M(\Sigma(t,\vartheta_0,\vartheta)-\Sigma(\vartheta_0,\vartheta))\bigr)dt
\right| \\
&\quad \le
\frac{\|\mathcal M\|}{T}\int_0^T
\sup_{(\vartheta_0,\vartheta)\in\bar\Theta^2}
\bigl\|\Sigma(t,\vartheta_0,\vartheta)-\Sigma(\vartheta_0,\vartheta)\bigr\|dt
\le \frac{C}{T}.
\end{aligned}
\end{equation}
Using the exponential decay in \ref{prop:aou-propositions} \eqref{item:aou-c} and Isserlis's theorem \citep{10.1093/biomet/12.1-2.134}, we can show that covariances of $h(Z_t(\vartheta_0,\vartheta))$ decay exponentially. Indeed, after expanding $h$, the mixed linear-quadratic terms vanish because $(Z_t,Z_s)$ is centred Gaussian, while the quadratic-quadratic terms are reduced by Isserlis's theorem to products of covariances between $Z_t$ and $Z_s$. Hence, all terms are controlled by covariances between $Z_t$ and $Z_s$, which decay exponentially by \ref{prop:aou-propositions} \eqref{item:aou-c}. Therefore, for a constant $C$ depending on $\mathcal M$ and $\mathcal K$, we obtain
\[\bigl|
  \operatorname{Cov}_{\vartheta_0}
  \bigl(
  h(Z_t(\vartheta_0,\vartheta)),
  h(Z_s(\vartheta_0,\vartheta))
  \bigr)
  \bigr|
  \le Ce^{-\beta |t-s|}.
\]
Hence,
\begin{equation}\label{eq:variance}
\begin{aligned}
&\sup_{(\vartheta_0,\vartheta)\in\bar\Theta^2}
\operatorname{Var}_{\vartheta_0}
\left(
\frac{1}{T}\int_0^T h(Z_t(\vartheta_0,\vartheta))dt
\right)\le
\frac{1}{T^2}\int_0^T\int_0^T Ce^{-\beta |t-s|}dsdt
\le \frac{2C}{\beta T}.
\end{aligned}
\end{equation}
Combining \eqref{eq:bias} and \eqref{eq:variance} with the Cauchy--Schwarz inequality, we obtain, for a new constant $C<\infty$
\begin{equation}\label{eq:OU-moment-bound}
\sup_{(\vartheta_0,\vartheta)\in\bar{\Theta}^{2}}
\mathbb{E}_{\vartheta_0}\biggl[
\biggl|
\frac{1}{T}\int_0^T h\bigl(Z_t(\vartheta_0,\vartheta)\bigr)dt
-\mathbb E_{\mu_{\vartheta_0,\vartheta}}\bigl[h(\bar Z(\vartheta_0,\vartheta))\bigr]
\biggr|
\biggr]
\le C\Bigl(\frac{1}{\sqrt{T}}+\frac{1}{T}\Bigr)\xrightarrow[T\to \infty]{}0,
\end{equation}
which completes the proof.
\end{proof}
\subsection{Fisher Information and Identifiability}\label{subsec:FisherInformation_Derivation}
Next, we prove the mean square differentiability of the log-likelihood function, the uniform convergence of the Fisher information and the uniform separation property required for identifiability.

\begin{proof}[\textbf{Proof of Lemma \ref{lem:score-derivation}}]
\phantomsection\label{prf:score-derivation}
  Let $g_t(\vartheta)=\vartheta m_t(\vartheta)$. From Lemma \ref{lem:mean-square-diff}, Lemmas \ref{lem:m_exp_bounds} and \ref{lem:partial_m_bounds} and \eqref{eq:bound_rest}, we find
  \begin{equation}\label{eq:g-bounds}
    \E_{\vartheta_0}\left[(g_t(\vartheta+h)-g_t(\vartheta))^4\right]\le Ch^4, \quad \E_{\vartheta_0}\left[\left(\frac{g_t(\vartheta+h)-g_t(\vartheta)}{h}-\partial_{\vartheta}g_t(\vartheta)\right)^2\right]\le Ch^2.
  \end{equation}
  Expanding the $dX_t$ term under $\Prob_{\vartheta_0}$, we obtain
  \begin{align}
    \begin{split}
      \frac{\ell_T(\vartheta_0+h)-\ell_T(\vartheta_0)}{h}-S_T(\vartheta_0)
       & =\int_0^T\left(\frac{g_t(\vartheta_0+h)-g_t(\vartheta_0)}{h}-\partial_{\vartheta}g_t(\vartheta_0)\right)d\bar W_t \\
       & \quad-\int_0^T\frac{(g_t(\vartheta_0+h)-g_t(\vartheta_0))^2}{2h}dt\\
       &\eqqcolon A_h + B_h.
    \end{split}\label{eq:diffquot_minus_score_theta0_cancelled}
  \end{align}
  It suffices to show that
  $\E_{\vartheta_0}[A_h^2]\to 0$ and $\E_{\vartheta_0}[B_h^2]\to 0$ as $h\to 0$.
  We begin with the stochastic integral term $A_h$. By It\^o's isometry and \eqref{eq:g-bounds}, we have
\begin{equation*}
  \E_{\vartheta_0}[A_h^2] = \E_{\vartheta_0}\Bigl[\int_0^T \Bigl(\frac{g_t(\vartheta_0+h)-g_t(\vartheta_0)}{h}-\partial_{\vartheta}g_t(\vartheta_0)\Bigr)^2 dt\Bigr] \le \int_0^T C h^2 dt = CTh^2 \xrightarrow[h\to 0]{} 0.
\end{equation*}
  We next analyse the drift term.
  To control the drift $B_h$, we use the bound in \eqref{eq:g-bounds} to obtain
  \begin{align*}
    \E_{\vartheta_0}[B_h^2]
     &= \E_{\vartheta_0}\Bigg[\Big(\int_0^T \frac{\bigl(g_t(\vartheta_0+h)-g_t(\vartheta_0)\bigr)^2}{2h}dt\Big)^2\Bigg]\\
     &\le \frac{T}{4h^2}\int_0^T \E_{\vartheta_0}\left[\bigl(g_t(\vartheta_0+h)-g_t(\vartheta_0)\bigr)^4\right]dt\\
     &     \le \frac{T}{4h^2}\int_0^T Ch^4dt
    = \frac{CT^2}{4}h^2 \xrightarrow[h\to 0]{}0 .
  \end{align*}
  Combining the two previous bounds completes the proof.\qedhere

\end{proof}

\begin{proof}[\textbf{Proof of Proposition \ref{prop:uniform-fisher}}]
\phantomsection\label{prf:uniform-fisher}
We first show that $I(\vartheta_0)\in(0,\infty)$ for any $\vartheta_0\in K$.
  Recalling the expression for $\Sigma$ from \eqref{eq:Limiting_Covariance}, we obtain
  \[
    I(\vartheta_0)
    = v(\vartheta_0)^\top \Sigma(\vartheta_0)v(\vartheta_0)
    = \int_0^\infty \bigl\|B(\vartheta_0)^\top e^{A(\vartheta_0)^\top s}v(\vartheta_0)\bigr\|^2ds.
  \]
  From the expression for $B(t,\vartheta_0)$ in \eqref{eq:full_proc} and the convergence of $\gamma_t(\vartheta_0) \to \gamma_{\infty}(\vartheta_0)$ as $t\to\infty$ (compare Remark \ref{rem:on_filter} \eqref{num:gamma_conv}), we find that for $\vartheta_0\neq 0$ holds
  \begin{equation*}
    \bigl\|B(\vartheta_0)^\top v(\vartheta_0)\bigr\|^2
    = \frac{\vartheta_0^2}{c^2+\vartheta_0^2} \neq 0.
  \end{equation*}
  Consequently, the integrand is non-zero for $s$ in a neighbourhood of zero and therefore $I(\vartheta_0)>0$. For $\vartheta_0=0$, it holds that $I(0)=\E_{0}[m(0)^2]$, where $m(0)$ denotes the stationary limit of $m_t(0)$. The SDE for $m_t(0)$ from \eqref{eq:filter-eq} is driven by the noisy process $X_t$ through the drift term $bX_t$ with $b\neq 0$ (Assumption \ref{ass:param_space}). Hence, $m(0)$ is almost surely not a constant and $I(0)>0$. Moreover, $\vartheta_0\mapsto \Sigma(\vartheta_0)$ is continuous as a consequence of Proposition \ref{prop:aou-propositions}\ref{item:aou-b}. Since $K$ is compact, this ensures that $0<\inf_{\vartheta_0\in K}I(\vartheta_0)\le \sup_{\vartheta_0\in K}I(\vartheta_0)<\infty$.
  
  \noindent We proceed with the convergence statement. Define $h_t(\vartheta_0) \coloneqq \mathbb{E}_{\vartheta_0}\big[(m_t(\vartheta_0)+\vartheta_0\partial_\vartheta m_t(\vartheta_0))^2\big]$, which is the integrand of the Fisher information.
  The process $\big(X_t,m_t(\vartheta_0),\partial_\vartheta m_t(\vartheta_0)\big)_{t\ge 0}$ is centred and satisfies the assumptions of Proposition \ref{prop:aou-propositions} (see Remark \ref{rem:stable-process}). By Proposition \ref{prop:aou-propositions}\ref{item:aou-b} its covariance $\Sigma(t,\vartheta_0)$ converges uniformly on $K$ to the stationary covariance matrix $\Sigma_{\infty}(\vartheta_0)$ of the limiting process, such that
  \[
    h_t(\vartheta_0)
    =\tilde{v}(\vartheta_0)^\top\Sigma(t,\vartheta_0)\tilde{v}(\vartheta_0)
    \xrightarrow[t\to\infty]{} \tilde{v}(\vartheta_0)^\top\Sigma_{\infty}(\vartheta_0)\tilde{v}(\vartheta_0)=
    I(\vartheta_0)
  \]
  uniformly in $\vartheta_0\in K$ for $\tilde{v}(\vartheta_0)\coloneq(0,1,\vartheta_0)^\top$. Since, in addition, $\sup_{t\ge0,\vartheta_0\in K} h_t(\vartheta_0)<\infty$ by Lemmas \ref{lem:m_exp_bounds} and \ref{lem:partial_m_bounds}, the time averages of $h_t(\vartheta_0)$ converge uniformly to $I(\vartheta_0)$ for $\vartheta_0\in K$.
\end{proof}

\begin{lemma}\label{lem:obs_inf_conv}
We define the observed information at $\vartheta$ by
\begin{equation}
    J_T(\vartheta)=\int_0^T \bigl(m_t(\vartheta)+\vartheta\partial_{\vartheta}m_t(\vartheta)\bigr)^2dt.
\end{equation}
Then the observed information evaluated at the MLE satisfies
\begin{equation*}
    \frac{1}{T}J_T(\hat{\vartheta}_T)\xrightarrow[]{\Prob_{\vartheta_0}}I(\vartheta_0).
\end{equation*}
\end{lemma}
\begin{proof}
We first prove a Hölder bound for $\vartheta\mapsto T^{-1}J_T(\vartheta)$ with a random constant that is stochastically bounded. Write $g_t(\vartheta)=m_t(\vartheta)+\vartheta\partial_{\vartheta}m_t(\vartheta)$, then we have
\begin{align*}
\E_{\vartheta_0}\bigl[(J_T(\vartheta_1)-J_T(\vartheta_2))^2\bigr]
&=\E_{\vartheta_0}\biggl[\biggl(\int_0^T \bigl(g_t(\vartheta_1)-g_t(\vartheta_2)\bigr)\bigl(g_t(\vartheta_1)+g_t(\vartheta_2)\bigr)dt\biggr)^2\biggr]\\
&\le T\int_0^T \E_{\vartheta_0}\bigl[\bigl(g_t(\vartheta_1)-g_t(\vartheta_2)\bigr)^4\bigr]^{1/2}\E_{\vartheta_0}\bigl[\bigl(g_t(\vartheta_1)+g_t(\vartheta_2)\bigr)^4\bigr]^{1/2}dt \\
&\le CT^2|\vartheta_1-\vartheta_2|^2,
\end{align*}
from Lemmas \ref{lem:m_exp_bounds} and \ref{lem:partial_m_bounds} and the fact that $g_t(\vartheta_1)-g_t(\vartheta_2)$ and $g_t(\vartheta_1)+g_t(\vartheta_2)$ are both Gaussian. We can then apply the quadratic Garsia--Rodemich--Rumsey inequality \citep[Appendix~A, Cor.~A.2 with $q=2,\alpha=s$]{Friz_Victoir_2010} for $\frac{1}{T}J_T(\vartheta)$. 
For $s\in(1/2,1)$, this gives 
\begin{equation}
\biggl|\frac{1}{T}J_T(\vartheta_1)-\frac{1}{T}J_T(\vartheta_2)\biggr|\le K_s \mathcal J_T(s)^{1/2}|\vartheta_1-\vartheta_2|^{s-1/2}, \label{eq:auxHölder}
\end{equation} 
where 
\begin{equation*} \mathcal J_T(s)=\iint_{\Theta \times \Theta}\frac{\bigl|\frac{1}{T}J_T(u)-\frac{1}{T}J_T(v)\bigr|^2}{|u-v|^{1+2s}}dudv. 
\end{equation*} 
Moreover, since  $s<1$, we have
\begin{equation*} 
\E_{\vartheta_0}[\mathcal J_T(s)]\le C\iint_{\Theta\times\Theta}|u-v|^{1-2s}dudv<\infty,
\end{equation*} 
such that $K_s\mathcal J_T(s)^{1/2}$ is stochastically bounded. Now write 
\begin{equation} \frac{1}{T}J_T(\hat{\vartheta}_T)-I(\vartheta_0)=\frac{1}{T}J_T(\vartheta_0)-I(\vartheta_0)+\biggl(\frac{1}{T}J_T(\hat{\vartheta}_T)-\frac{1}{T}J_T(\vartheta_0)\biggr). \label{eq:aux_Hölder_Decomposition}
\end{equation} 
The first term in \eqref{eq:aux_Hölder_Decomposition} converges to $0$ in $\Prob_{\vartheta_0}$ by Proposition~\ref{prop:aou-propositions}\ref{item:aou-a}. For the second term, the Hölder bound \eqref{eq:auxHölder} gives
\begin{equation*}
    \biggl|\frac{1}{T}J_T(\hat{\vartheta}_T)-\frac{1}{T}J_T(\vartheta_0)\biggr|
    \le K_s\mathcal J_T(s)^{1/2}|\hat{\vartheta}_T-\vartheta_0|^{s-1/2}.
\end{equation*}
Since $K_s\mathcal J_T(s)^{1/2}=O_{\Prob_{\vartheta_0}}(1)$ is stochastically bounded and the MLE is consistent, the second term in \eqref{eq:aux_Hölder_Decomposition} converges to $0$ in $\Prob_{\vartheta_0}$ and we conclude the proof.
\end{proof}

\label{sec:ident}
\begin{proposition}\label{prop:uniform-separation}
  For $\lambda>0$, there exist constants $K(\lambda)>0$ and $T_0(\lambda)<\infty$ such that for all $t\ge T_0(\lambda)$,
  \[
    \inf_{\vartheta_0\in \Theta} \inf_{\substack{\vartheta\in \Theta\\ |\vartheta-\vartheta_0|\ge \lambda}}
    \mathbb{E}_{\vartheta_0}\left[(\vartheta m_t(\vartheta)-\vartheta_0 m_t(\vartheta_0))^2\right]
    \ge K(\lambda).
  \]
\end{proposition}

\begin{proof} For $(\vartheta_0,\vartheta)\in\bar\Theta^2$, let
  $\bigl(m(\vartheta_0),m(\vartheta)\bigr)$ denote a random vector distributed
  according to the stationary marginal law associated with $\bigl(m_t(\vartheta_0),m_t(\vartheta)\bigr)$ as $t\to\infty$, whose existence is ensured by Proposition \ref{prop:aou-propositions} and Remark \ref{rem:stable-process}. Let $g_t(\vartheta)\coloneq\vartheta m_t(\vartheta)$, $g(\vartheta)\coloneq\vartheta m(\vartheta)$ and
  define \[
    \mathcal A_\lambda \coloneq\{(\vartheta_0,\vartheta)\in \bar{\Theta}^2: |\vartheta-\vartheta_0|\ge \lambda\}.
  \]
  Set $F_t(\vartheta_0,\vartheta)\coloneq\mathbb{E}_{\vartheta_0}\left[(g_t(\vartheta)-g_t(\vartheta_0))^2\right]$ for $(\vartheta_0,\vartheta)\in\bar{\Theta}^2$. From Remark \ref{rem:stable-process} we know that $\bigl(X_t,\allowbreak m_t(\vartheta_0),\allowbreak m_t(\vartheta)\bigr)$ satisfies the assumptions of Proposition \ref{prop:aou-propositions}, such that \eqref{item:aou-b} of this proposition yields the uniform convergence of the covariance matrix of the process $\bigl(X_t,\allowbreak m_t(\vartheta_0),\allowbreak m_t(\vartheta)\bigr)$ to the covariance matrix of the limiting process. Since $F_t(\vartheta_0,\vartheta)$ is a quadratic function of the covariance, there exists a continuous limit function on $\bar{\Theta}^2$ such that
    \begin{equation}\label{eq:uniform-conv}
    \sup_{(\vartheta_0,\vartheta)\in \bar{\Theta}^2}
    \bigl|F_t(\vartheta_0,\vartheta)-F(\vartheta_0,\vartheta)\bigr|
    \xrightarrow[t\to\infty]{}0,\quad F(\vartheta_0,\vartheta)
    \coloneq\mathbb{E}_{\vartheta_0}\left[(g(\vartheta_0)-g(\vartheta))^2\right].
  \end{equation}
  We first show that, for fixed $\vartheta_0$, the map
  $\vartheta\mapsto -\E_{\vartheta_0}[(g(\vartheta)-g(\vartheta_0))^2]$
  is uniquely maximised at $\vartheta_0$. At $\vartheta=\vartheta_0$ the value is zero, so $\vartheta_0$ is a maximiser. If $\bar\vartheta$ were another maximiser, then
  $\mathbb E_{\vartheta_0}[(g(\bar\vartheta)-g(\vartheta_0))^2]=0$ and $\bar\vartheta m(\bar\vartheta)=\vartheta_0 m(\vartheta_0)$ almost surely.

  If $\vartheta_0\neq 0$, then $m(\vartheta_0)=\bar\vartheta/\vartheta_0m(\bar\vartheta)$ almost surely.
  Plugging this relation into the limiting ($t \to \infty$) filter equation from \eqref{eq:filter-eq} yields the identities $\bar\vartheta^2=\vartheta_0^2$ and $b=b\bar\vartheta/\vartheta_0$. Since $b\neq 0$ from Assumption \ref{ass:param_space},
  this forces $\bar\vartheta=\vartheta_0$.

  If $\vartheta_0=0$, the equality reduces to $\bar\vartheta m(\bar\vartheta)=0$ almost surely. By \eqref{eq:filter-eq}, the limiting SDE for $m_t(\bar\vartheta)$ has a nonzero diffusion coefficient whenever $\bar\vartheta\neq 0$. Hence, $m(\bar\vartheta)$ cannot be almost surely zero for $\bar\vartheta\neq 0$. Therefore, we find $\bar\vartheta=0=\vartheta_0$.

  Thus, for $\vartheta_0\neq \vartheta$, we obtain $F(\vartheta_0,\vartheta)=\mathbb{E}_{\vartheta_0}\left[(g(\vartheta_0)-g(\vartheta))^2\right]>0$.
  Since $F$ is continuous and $\mathcal A_\lambda$ is compact (a closed bounded set), $F$ attains a strictly positive minimum on $\mathcal A_\lambda$, namely $K_0(\lambda)\coloneq\inf_{(\vartheta_0,\vartheta)\in \mathcal A_\lambda} F(\vartheta_0,\vartheta)>0$. 
  
  By the uniform convergence stated in \eqref{eq:uniform-conv}, there exists $T_0(\lambda)$ such that, for all $t\ge T_0(\lambda)$ holds
  \[
    \sup_{(\vartheta_0,\vartheta)\in \mathcal A_\lambda}
    \bigl|F_t(\vartheta_0,\vartheta)-F(\vartheta_0,\vartheta)\bigr| \le \tfrac12 K_0(\lambda).
  \]
  Therefore, for all $(\vartheta_0,\vartheta)\in \mathcal A_\lambda$ and $t\ge T_0(\lambda)$, we obtain
  \[
    F_t(\vartheta_0,\vartheta) \ge F(\vartheta_0,\vartheta)-\tfrac12 K_0(\lambda) \ge \tfrac12 K_0(\lambda).
  \]
  Setting $K(\lambda)\coloneq\tfrac12 K_0(\lambda)$ yields
  \[
    \inf_{(\vartheta_0,\vartheta)\in \mathcal A_\lambda} F_t(\vartheta_0,\vartheta) \ge K(\lambda),
    \qquad t\ge T_0(\lambda).
  \]
  Since $\{(\vartheta_0,\vartheta)\in \Theta^2: |\vartheta-\vartheta_0|\ge\lambda\}\subset \mathcal A_\lambda$, the same lower bound holds for
  \[
    \inf_{\vartheta_0\in \Theta} \inf_{\substack{\vartheta\in\Theta\\|\vartheta-\vartheta_0|\ge\lambda}} F_t(\vartheta_0,\vartheta)
  \]
  and the proof is completed.
\end{proof}

\noindent\textbf{Information Ratio.} \phantomsection\label{sec:information_ratio_conv}
We briefly recall the statistics of the fully observed system $(X^T, Y^T)$. Since the parameter $\vartheta$ only appears in the drift of $X_t$, Theorem 7.19 of \citet{LiptserShiryaev2001} yields the full-data log-likelihood 
\begin{equation*}
    \log L_{T,X,Y}(\vartheta) = \vartheta \int_0^T Y_t(dX_t  - a X_t dt) - \frac{1}{2}\vartheta^2\int_0^T Y_t^2 dt. 
\end{equation*}
Differentiating with respect to $\vartheta$ and equating to zero provides the full-data MLE
\begin{equation*}
    \hat{\vartheta}_{T,X,Y} = \frac{\int_0^T Y_t dX_t - a\int_0^T X_t Y_t dt}{\int_0^T Y_t^2 dt}. 
\end{equation*}
The corresponding full-data score function under $\mathbb{P}_{\vartheta_0}$ is given by $S_{X,Y,T}(\vartheta_0) = \int_0^T Y_t dW_t^X$, which immediately gives the full-data Fisher information $I_{X,Y,T}(\vartheta_0) = \mathbb{E}_{\vartheta_0}[\int_0^T Y_t^2 dt]$. 

Solving \eqref{eq:lyapunov-covariance} for $\Sigma(\vartheta_0)$, we obtain the following limits.

\noindent (i) As $|\vartheta_0|\to\infty$, we have
\begin{align*}
  \bigl[\Sigma(\vartheta_0)\bigr]_{2,2}\to-\frac{1}{2(a+c)},\qquad
  \bigl[\Sigma(\vartheta_0)\bigr]_{3,3}\to-\frac{1}{2(a+c)}, \\
  \bigl[\Sigma(\vartheta_0)\bigr]_{3,4}\to\frac{1}{2(a+c)\vartheta_0},\qquad
  \bigl[\Sigma(\vartheta_0)\bigr]_{4,4}\to-\frac{1}{2(a+c)\vartheta_0^{2}}.
\end{align*}
Substituting these into \eqref{eq:information_repr} yields
\[
\begin{aligned}
  I_X(\vartheta_0)
  &=
  \bigl[\Sigma(\vartheta_0)\bigr]_{3,3}
  + 2\vartheta_0 \bigl[\Sigma(\vartheta_0)\bigr]_{3,4}
  + \vartheta_0^2 \bigl[\Sigma(\vartheta_0)\bigr]_{4,4}
  \to 0, \\
  I_{X,Y}(\vartheta_0)
  &=\bigl[\Sigma(\vartheta_0)\bigr]_{2,2}
  \to -\frac{1}{2(a+c)},
\end{aligned}
\]
as $|\vartheta_0|\to\infty$.
Hence, we obtain
\[
  R(\vartheta_0)=\frac{I_X(\vartheta_0)}{I_{X,Y}(\vartheta_0)}\to0,\quad |\vartheta_0|\to\infty.
\]

\noindent (ii) Note that
\[
  \bigl[\Sigma(\vartheta_0)\bigr]_{2,2}
  =\frac{b}{2(a+c)\vartheta_0}-\frac{1}{2c},\qquad
  \bigl[\Sigma(\vartheta_0)\bigr]_{3,3}
  =\frac{b}{2(a+c)\vartheta_0}-\frac{\gamma_\infty^2(\vartheta_0)\vartheta_0^2}{2c},
\]
such that $\bigl[\Sigma(\vartheta_0)\bigr]_{2,2}$ and $\bigl[\Sigma(\vartheta_0)\bigr]_{3,3}$ grow linearly in $b$, while
$\bigl[\Sigma(\vartheta_0)\bigr]_{2,2}-\bigl[\Sigma(\vartheta_0)\bigr]_{3,3}$ remains constant as |$b|\to\infty$. Moreover, by \eqref{eq:lyapunov-covariance}, $\bigl[\Sigma(\vartheta_0)\bigr]_{3,4}$ and $\bigl[\Sigma(\vartheta_0)\bigr]_{4,4}$ converge to finite constants independent of $b$ as well. We conclude that
\[
  R(\vartheta_0)=\frac{I_X(\vartheta_0)}{I_{X,Y}(\vartheta_0)}\to1
  \quad |b|\to\infty.
\]

\noindent (iii) As $c\to-\infty$, using \eqref{eq:lyapunov-covariance} we obtain
\[
  \bigl[\Sigma(\vartheta_0)\bigr]_{2,2}
  =-\frac{1}{2c}-\frac{b^2}{2ac^2}+O\left(\frac{1}{c^{3}}\right) \to 0,\qquad
  \bigl[\Sigma(\vartheta_0)\bigr]_{3,3}
  =-\frac{b^2}{2ac^2}+O\left(\frac{1}{c^{3}}\right) \to 0.
\]
We also find that
$\bigl[\Sigma(\vartheta_0)\bigr]_{3,4}=\mathcal{O}(c^{-3})$ and $\bigl[\Sigma(\vartheta_0)\bigr]_{4,4}=\mathcal{O}(c^{-3})$. Hence, we obtain
\[
  I_X(\vartheta_0)=O\left(\frac{1}{c^{2}}\right),
  \qquad
  I_{X,Y}(\vartheta_0)=O\left(\frac{1}{|c|}\right),
\]
and
\[
  R(\vartheta_0)=\frac{I_X(\vartheta_0)}{I_{X,Y}(\vartheta_0)}
  =O\left(\frac{1}{|c|}\right)\xrightarrow[c\to-\infty]{} 0.
\]

\noindent (iv) As $a\to-\infty$, using \eqref{eq:lyapunov-covariance} we obtain
\[
  \bigl[\Sigma(\vartheta_0)\bigr]_{2,2}
  \to-\frac{1}{2c},\qquad
  \bigl[\Sigma(\vartheta_0)\bigr]_{3,3}
  \to-\frac{\vartheta_0^2\gamma_\infty^2(\vartheta_0)}{2c},
\]
and
\[
  \bigl[\Sigma(\vartheta_0)\bigr]_{3,4}
  \to
  \frac{ -\dfrac{\vartheta_0^3\gamma_\infty^3(\vartheta_0)}{2c}
    -\vartheta_0^{3}\gamma_\infty(\vartheta_0)\Bigl(\gamma_\infty(\vartheta_0)+\gamma_\infty'(\vartheta_0)\Bigr)+b \Bigl(\gamma_\infty(\vartheta_0)+\vartheta_0\gamma_\infty'(\vartheta_0)\Bigr) }
  {2c-\vartheta_0^{2}\gamma_\infty(\vartheta_0)}.
\]
Putting $\bigl[\Sigma(\vartheta_0)\bigr]_{3,4}$ into the equation of $\bigl[\Sigma(\vartheta_0)\bigr]_{4,4}$ in \eqref{eq:lyapunov-covariance}, we find that $\bigl[\Sigma(\vartheta_0)\bigr]_{4,4}$ also converges to some constant. Hence, $I_X(\vartheta_0)$ and $I_{X,Y}(\vartheta_0)$ both converge to constants depending on $b$, $c$ and $\vartheta_0$ and the ratio converges to some constant $\rho\in(0,1)$.

\section*{Acknowledgments}
The authors would like to thank Markus Reiß for many fruitful discussions.

\section*{Funding}
S.G. acknowledges funding by the
Deutsche Forschungsgemeinschaft (DFG, German Research Foundation) – CRC/TRR 388
"Rough Analysis, Stochastic Dynamics and Related Fields" – Project ID 516748464.

\noindent H.M.G. was funded by the Deutsche Forschungsgemeinschaft
(DFG, German Research Foundation) -- SFB 1294,
Project A01 "Statistics for Stochastic Partial Differential Equations" -- Project-ID 318763901.

\bibliographystyle{plainnat}
\bibliography{references}

@article{blathNewCoalescentSeedbank2016a,
  title = {A New Coalescent for Seed-Bank Models},
  author = {Blath, Jochen and Gonz{\'a}lez Casanova, Adri{\'a}n and Kurt, Noemi and {Wilke-Berenguer}, Maite},
  year = {2016},
  journal = {The Annals of Applied Probability},
  volume = {26},
  number = {2},
  issn = {1050-5164},
  doi = {10.1214/15-AAP1106},
}

@book{klenkeWahrscheinlichkeitstheorie2020,
  title = {{Wahrscheinlichkeitstheorie}},
  author = {Klenke, Achim},
  year = {2020},
  publisher = {Springer},
  address = {Berlin Heidelberg},
  doi= {10.1007/978-3-662-62089-2},
}

@article{kurisakiParameterEstimationErgodic2023,
  title = {Parameter Estimation for Ergodic Linear {{SDEs}} from Partial and Discrete Observations},
  author = {Kurisaki, Masahiro},
  year = {2023},
  journal = {Statistical Inference for Stochastic Processes},
  volume = {26},
  number = {2},
  pages = {279--330},
  issn = {1387-0874, 1572-9311},
  doi = {10.1007/s11203-023-09288-w},
}

@article{kutoyantsParameterEstimationHidden2021,
author = {Yury A. Kutoyants and Li Zhou},
title = {{On parameter estimation of the hidden Gaussian process in perturbed SDE.}},
volume = {15},
journal = {Electronic Journal of Statistics},
number = {1},
publisher = {Institute of Mathematical Statistics and Bernoulli Society},
pages = {211 -- 234},
year = {2021},
doi = {10.1214/20-EJS1788},
}

@article{Huebner1995,
  title = {On Asymptotic Properties of Maximum Likelihood Estimators for Parabolic Stochastic PDEs},
  author = {Huebner, M. and Rozovskii, Boris L.},
  year = {1995},
  journal = {Probability Theory and Related Fields},
  volume = {103},
  pages = {143--163},
}

@book{murrayMathematicalBiologyII2003,
  title = {Mathematical {{Biology}}: {{II}}: {{Spatial Models}} and {{Biomedical Applications}}},
  shorttitle = {Mathematical {{Biology}}},
  editor = {Murray, J. D.},
  year = {2003},
  series = {Interdisciplinary {{Applied Mathematics}}},
  volume = {18},
  publisher = {Springer New York},
  address = {New York, NY},
  doi = {10.1007/b98869},
  urldate = {2026-05-06},
  copyright = {http://www.springer.com/tdm},
  isbn = {978-0-387-95228-4 978-0-387-22438-1},
}

@article{trottnerConcentrationAnalysisMultivariate2023,
  title = {Concentration Analysis of Multivariate Elliptic Diffusions},
  author = {Trottner, Lukas and {Aeckerle-Willems}, Cathrine and Strauch, Claudia},
  year = {2023},
  journal = {Journal of Machine Learning Research},
  volume = {24},
  number = {106},
  pages = {1--38},
  doi = {10.5555/3648699.3648805},
}

@article{aeckerle-willemsConcentrationScalarErgodic2021,
  title = {Concentration of Scalar Ergodic Diffusions and Some Statistical Implications},
  author = {{Aeckerle-Willems}, Cathrine and Strauch, Claudia},
  year = {2021},
  journal = {Annales de l'Institut Henri Poincar{\'e}, Probabilit{\'e}s et Statistiques},
  volume = {57},
  number = {4},
  pages = {1857--1887},
  doi = {10.1214/20-AIHP1144},
}

@article{dalalyanAsymptoticStatisticalEquivalence2007,
  title = {Asymptotic Statistical Equivalence for Ergodic Diffusions: The Multidimensional Case},
  author = {Dalalyan, Arnak and Rei{\ss}, Markus},
  year = {2007},
  journal = {Probability Theory and Related Fields},
  volume = {137},
  pages = {25--47},
  doi = {10.1007/s00440-006-0502-7},
}

@article{kutoyantsParameterEstimationHidden2019,
  title = {On Parameter Estimation of Hidden Ergodic {{Ornstein-Uhlenbeck}} Process},
  author = {Kutoyants, Yury A.},
  year = {2019},
  journal = {Electronic Journal of Statistics},
  volume = {13},
  number = {2},
  issn = {1935-7524},
  doi = {10.1214/19-EJS1631},
}

@article{fitzhughImpulsesPhysiologicalStates1961,
  title = {Impulses and {{Physiological States}} in {{Theoretical Models}} of {{Nerve Membrane}}},
  author = {FitzHugh, Richard},
  year = {1961},
  journal = {Biophysical Journal},
  volume = {1},
  number = {6},
  pages = {445--466},
  issn = {00063495},
  doi = {10.1016/S0006-3495(61)86902-6},
}

@article{lennonPrinciplesSeedBanks2021,
  title = {Principles of Seed Banks and the Emergence of Complexity from Dormancy},
  author = {Lennon, Jay T. and Den Hollander, Frank and {Wilke-Berenguer}, Maite and Blath, Jochen},
  year = {2021},
  journal = {Nature Communications},
  volume = {12},
  number = {1},
  pages = {4807},
  issn = {2041-1723},
  doi = {10.1038/s41467-021-24733-1},
}

@article{Altmeyer2022,
author = {Altmeyer, Randolf and Bretschneider, Till and Jan\'{a}k, Josef and Rei\ss{}, Markus},
title = {Parameter Estimation in an SPDE Model for Cell Repolarization},
journal = {SIAM/ASA Journal on Uncertainty Quantification},
volume = {10},
number = {1},
pages = {179-199},
year = {2022},
doi = {10.1137/20M1373347},
}

@article{budhirajaAsymptoticStabilityErgodicity2003,
  title = {Asymptotic Stability, Ergodicity and Other Asymptotic Properties of the Nonlinear Filter},
  author = {Budhiraja, A},
  year = {2003},
  journal = {Annales de l'Institut Henri Poincare (B) Probability and Statistics},
  volume = {39},
  number = {6},
  pages = {919--941},
  issn = {02460203},
  doi = {10.1016/S0246-0203(03)00022-0},
}

@book{cappeInferenceHiddenMarkov2005,
  title = {Inference in {{Hidden Markov Models}}},
  author = {Capp{\'e}, Olivier and Moulines, Eric and Ryd{\'e}n, Tobias},
  year = {2005},
  series = {Springer {{Series}} in {{Statistics}}},
  publisher = {Springer New York},
  address = {New York, NY},
  doi = {10.1007/0-387-28982-8},
  urldate = {2026-05-06},
  copyright = {http://www.springer.com/tdm},
  isbn = {978-0-387-40264-2 978-0-387-28982-3},
}

@article{demboParameterEstimationPartially1986,
  title = {Parameter Estimation of Partially Observed Continuous Time Stochastic Processes via the {{EM}} Algorithm},
  author = {Dembo, A. and Zeitouni, O.},
  year = {1986},
  journal = {Stochastic Processes and their Applications},
  volume = {23},
  number = {1},
  pages = {91--113},
  issn = {03044149},
  doi = {10.1016/0304-4149(86)90018-9},
}

@article{elliottExactFiniteDimensionalFilters1997,
  title = {Exact {{Finite-Dimensional Filters}} for {{Maximum Likelihood Parameter Estimation}} of {{Continuous-time Linear Gaussian Systems}}},
  author = {Elliott, Robert J. and Krishnamurthy, Vikram},
  year = {1997},
  journal = {SIAM Journal on Control and Optimization},
  volume = {35},
  number = {6},
  pages = {1908--1923},
  issn = {0363-0129, 1095-7138},
  doi = {10.1137/S036301299529255X},
}

@article{mongilloOnlineLearningHidden2008,
  title = {Online {{Learning}} with {{Hidden Markov Models}}},
  author = {Mongillo, Gianluigi and Deneve, Sophie},
  year = {2008},
  journal = {Neural Computation},
  volume = {20},
  number = {7},
  pages = {1706--1716},
  issn = {0899-7667, 1530-888X},
  doi = {10.1162/neco.2008.10-06-351},
}

@article{papavasiliouParameterEstimationAsymptotic2006,
  title = {Parameter Estimation and Asymptotic Stability in Stochastic Filtering},
  author = {Papavasiliou, Anastasia},
  year = {2006},
  journal = {Stochastic Processes and their Applications},
  volume = {116},
  number = {7},
  pages = {1048--1065},
  issn = {03044149},
  doi = {10.1016/j.spa.2006.01.002},
}

@article{moura1986,
  title = {Identification and filtering--optimal recursive maximum likelihood approach},
  author = {Moura, Jos{\'e} M. F. and Mitter, S. K.},
  year = {1986},
  journal = {Laboratory for Information and Decision Systems, Massachusetts Institute of Technology},
}

@article{suraceOnlineMaximumLikelihoodEstimation2019,
  title = {Online {{Maximum-Likelihood Estimation}} of the {{Parameters}} of {{Partially Observed Diffusion Processes}}},
  author = {Surace, Simone Carlo and Pfister, Jean-Pascal},
  year = {2019},
  journal = {IEEE Transactions on Automatic Control},
  volume = {64},
  number = {7},
  pages = {2814--2829},
  issn = {0018-9286, 1558-2523, 2334-3303},
  doi = {10.1109/TAC.2018.2880404},
}

@book{LiptserShiryaev2001,
  author    = {Liptser, R. S. and Shiryaev, A. N.},
  title     = {Statistics of Random Processes I: General Theory},
  edition   = {2nd},
  publisher = {Springer-Verlag},
  year      = {2001},
  address   = {Berlin, Heidelberg},
  isbn      = {978-3-642-62890-0}
}

@article{KUTOYANTS2019248,
title = {On parameter estimation of the hidden Ornstein–Uhlenbeck process},
journal = {Journal of Multivariate Analysis},
volume = {169},
pages = {248-263},
year = {2019},
issn = {0047-259X},
doi = {10.1016/j.jmva.2018.09.008},
author = {Yury A. Kutoyants},
}

@book{IbragimovHasminskii1981,
  author    = {Ibragimov, I. A. and Has'minskii, R. Z.},
  title     = {Statistical Estimation: Asymptotic Theory},
  series    = {Applications of Mathematics},
  volume    = {16},
  publisher = {Springer-Verlag},
  address   = {New York},
  year      = {1981},
  isbn      = {0387905235},
  note      = {Translated by Samuel Kotz},
  mrnumber  = {0620321}
}

@article{KALLIANPUR1991284,
title = {Parameter estimation in linear filtering},
journal = {Journal of Multivariate Analysis},
volume = {39},
number = {2},
pages = {284-304},
year = {1991},
issn = {0047-259X},
doi = {10.1016/0047-259X(91)90102-8},
author = {G Kallianpur and R.S Selukar},
}

@book{kutoyants2004statistical,
  title     = {Statistical Inference for Ergodic Diffusion Processes},
  author    = {Kutoyants, Yuri A.},
  series    = {Springer Series in Statistics},
  year      = {2004},
  publisher = {Springer},
  address   = {New York},
  isbn      = {0-387-40225-6},
  doi       = {10.1007/978-0-387-40226-1}
}

@book{halmos1978bounded,
  title={Bounded Integral Operators on L2 Spaces},
  author={Halmos, Paul Richard and Sunder, V.S.},
  year={1978},
  publisher={Springer},
  series={Ergebnisse der Mathematik und ihrer Grenzgebiete},
  doi={10.1007/978-3-642-67016-9}
}

@article{Schur1911,
author = {Schur, J.},
journal = {Journal für die reine und angewandte Mathematik},
pages = {1-28},
title = {Bemerkungen zur Theorie der beschränkten Bilinearformen mit unendlich vielen Veränderlichen.},
volume = {140},
year = {1911},
}

@book{Friz_Victoir_2010, place={Cambridge}, series={Cambridge Studies in Advanced Mathematics}, title={Multidimensional Stochastic Processes as Rough Paths: Theory and Applications}, publisher={Cambridge University Press}, author={Friz, Peter K. and Victoir, Nicolas B.}, year={2010}, collection={Cambridge Studies in Advanced Mathematics}}

@book{van-der-Vaart1996,
  author = {van der Vaart, A. and Wellner, J.A.},
  title = {Weak Convergence and Empirical Processes. With Applications to Statistics},
  year = {1996},
  publisher = {Springer New York},
}

@book{Teschl2012-ODE,
  author    = {Gerald Teschl},
  title     = {Ordinary Differential Equations and Dynamical Systems},
  series    = {Graduate Studies in Mathematics},
  volume    = {140},
  publisher = {American Mathematical Society},
  address   = {Providence, RI},
  year      = {2012},
  doi       = {10.1090/gsm/140},
  isbn      = {978-0-8218-8328-0}
}

@book{khalil2002nonlinear,
  title     = {Nonlinear Systems},
  author    = {Khalil, Hassan K.},
  edition   = {3},
  year      = {2002},
  publisher = {Prentice Hall},
  address   = {Upper Saddle River, NJ}
}

@book{Kutoyants1984,
  author    = {Kutoyants, Yu. A.},
  title     = {Parameter Estimation for Stochastic Processes},
  year      = {1984},
  publisher = {Heldermann Verlag},
  address   = {Berlin},
  series    = {Research and Exposition in Mathematics},
  volume    = {6},
  isbn      = {3-88538-206-7},
  note      = {Translated and edited by B. L. S. Prakasa Rao}
}

@article{Chigansky_2008,
   title={Maximum likelihood estimator for hidden Markov models in continuous time},
   volume={12},
   ISSN={1572-9311},
   DOI={10.1007/s11203-008-9025-4},
   number={2},
   journal={Statistical Inference for Stochastic Processes},
   publisher={Springer Science and Business Media LLC},
   author={Chigansky, Pavel},
   year={2008},
   month={July},
   pages={139--163}
}

@Article{Brouste_2010,
journal={Statistical Inference for Stochastic Processes},
author={Alexandre Brouste and Marina Kleptsyna},
title={Asymptotic properties of MLE for partially observed fractional diffusion system},
year={2010},
month={April},
pages={1-13},
volume={13},
number={1},
doi={10.1007/s11203-009-9035-x}
}

@book{Da-Prato_Zabczyk_2014, 
place={Cambridge}, 
edition={2}, 
series={Encyclopedia of Mathematics and its Applications}, 
title={Stochastic Equations in Infinite Dimensions}, 
publisher={Cambridge University Press}, 
author={Da Prato, Giuseppe and Zabczyk, Jerzy}, 
year={2014}, 
collection={Encyclopedia of Mathematics and its Applications},
}

@book{Bain2009,
  title = {Fundamentals of Stochastic Filtering},
  author = {Bain, Alan and Crisan, Dan},
  year = {2009},
  publisher = {Springer New York},
  series = {Stochastic Modelling and Applied Probability},
  isbn = {9780387768960},
}

@book{LeCam1986,
  author    = {Le Cam, Lucien},
  title     = {Asymptotic Methods in Statistical Decision Theory},
  series    = {Springer Series in Statistics},
  publisher = {Springer},
  address   = {New York, NY},
  year      = {1986},
  doi       = {10.1007/978-1-4612-4946-7},
  isbn      = {978-0-387-96307-5},
  pages     = {XXVI, 742},
  edition   = {1}
}

@article{10.1093/biomet/12.1-2.134,
    author = {Isserlis, L.},
    title = {On a Formula for the Product-Moment Coefficient of Any Order of a Normal Frequency Distribution in Any Number of Variables},
    journal = {Biometrika},
    volume = {12},
    number = {1-2},
    pages = {134-139},
    year = {1918},
    month = {11},
    issn = {0006-3444},
    doi = {10.1093/biomet/12.1-2.134},
}

@article{bishop2017,
author = {Bishop, Adrian N. and Del Moral, Pierre},
title = {On the Stability of Kalman--Bucy Diffusion Processes},
journal = {SIAM Journal on Control and Optimization},
volume = {55},
number = {6},
pages = {4015-4047},
year = {2017},
doi = {10.1137/16M1102707},
}

\end{document}